\documentclass{amsart} 
\usepackage{xcolor}
\usepackage[english]{babel} 

\usepackage{amsmath,amsfonts,amsthm,amssymb} 
\usepackage{fancyhdr} 
\usepackage{verbatim} 
\usepackage{hyperref} 
\hypersetup{          
	colorlinks=true, linkcolor=blue, filecolor=magenta, urlcolor=blue,
	citecolor=magenta, linktoc=all, }

\usepackage[nameinlink,capitalise]{cleveref}
\Crefformat{enumi}{#2\textup{#1}#3}
\usepackage{bm} 
\usepackage{graphicx} 
\usepackage[font=footnotesize,labelfont=bf]{caption} 
\usepackage{tikz-cd} 
\usepackage{tikz} 
\tikzset{
	MyPersp/.style={scale=1.8,x={(-2cm,-0.5cm)},y={(2cm,-0.5cm)},
		z={(0,1cm)}},
	MyPoints/.style={fill=white,draw=black,thick}
}
\usetikzlibrary{intersections,external,calc,3d,patterns,decorations.markings}

\usepackage{float} 

\numberwithin{equation}{section} 
\numberwithin{figure}{section} 
\numberwithin{table}{section} 

\DeclareMathOperator{\scn}{sc}
\DeclareMathOperator\im{Im}

\title[On a quadratic form associated with  a surface automorphism]{On a quadratic form associated with a surface automorphism and its applications to Singularity Theory} 
\author{L. Alan\'is-L\'opez, E. Artal, C. Bonatti, X. G\'omez-Mont, M. Gonz\'alez Villa \& P. 
	Portilla} 
\thanks{EA is partially supported by 
Grant PID2020-114750GB-C31 funded by MCIN/AEI/10.13039/501100011033
and 
Departamento de Ciencia, Universidad y Sociedad del Conocimiento of the Gobierno de Arag{\'o}n
(E22\_20R: ``{\'A}lgebra y Geometr{\'i}a''); MGV and PPC are partially supported by 
Grant PID2020-114750GB-C33 funded by MCIN/AEI/10.13039/501100011033; CB is partially supported by the project ANR-19-CE40-0007; and XGM and PPC is partially supported by  CONACYT grant 286447.}

\date{\normalsize\today} 

\address{Departamento de Matem\'aticas, Universidad Aut\'onoma de Nuevo Le\'on, Av. Universidad s/n. Ciudad Universitaria San Nicol\'as de los Garza, Nuevo Le\'on, C.P. 66451, M\'exico}
\email{lilia.alanislpz@uanl.edu.mx}

\address{Departamento de Matem\'aticas-IUMA, Universidad de Zaragoza, c. Pedro Cerbuna 12, 50009 Zarag\-oza, Spain}
\email{artal@unizar.es}

\address{CNRS \& Institut de Math\'ematiques de Bourgogne (IMB, UMR CNRS 5584), Universit\'e Bourgogne, 9 av. Alain Savary, 21000 Dijon, France.}
\email{bonatti@u-bourgogne.fr}

\address{Centro de investigaci\'on en 
Matem\'aticas\\
Apartado Postal
402 \\
36000 Guanajuato, GTO. México}
\email{gmont@cimat.mx, manuel.gonzalez@cimat.mx, pablo.portilla@cimat.mx}

\usepackage{enumitem}

\usepackage{aliascnt}

\usepackage[all]{xy}

\usepackage{pifont}

\usepackage{setspace}
\usepackage[refpage,intoc]{nomencl}

\usepackage{tikz-cd} 

\usepackage{hyperref}

\newtheorem{thmA}{Theorem}

\newtheorem{thm}[equation]{Theorem}
\Crefname{thm}{Theorem}{thm}

\newtheorem{lemma}[equation]{Lemma}

\newtheorem{prop}[equation]{Proposition}
\Crefname{prop}{Proposition}{thm}

\newtheorem{cor}[equation]{Corollary}

\theoremstyle{definition}

\newtheorem{example}[equation]{Example}

\newtheorem{definition}[equation]{Definition}

\Crefname{notation}{Notation}{notation}

\theoremstyle{remark}
\newtheorem{remark}[equation]{Remark}

\DeclareMathOperator{\lcm}{lcm} 
\DeclareMathOperator\res{res}

\newcommand{\RR}{\mathbb{R}}
\newcommand{\QQ}{\mathbb{Q}}
\newcommand{\CC}{\mathbb{C}}
\newcommand{\C}{\mathbb{C}}
\newcommand{\NN}{\mathbb{N}}

\newcommand{\ZZ}{\mathbb{Z}}

\newcommand{\G}{\Gamma}
\newcommand{\MS}{\mathbb{S}}

\newcommand{\gss}{\Gamma_{\!\textrm{ss}}}
\newcommand{\Vss}{\mathcal{V}_{\textrm{ss}}}

\newcommand{\id}{\mathop{\mathrm{id}}\nolimits}

\newcommand{\calA}{\mathcal{A}}
\newcommand{\calC}{\mathcal{C}}
\newcommand{\calD}{\mathcal{D}}

\newcommand{\ii}{(ii)}

\renewcommand{\epsilon}{\varepsilon}

\newcounter{dummy}

\newcommand{\modgr}{\mathrm{Mod}_{g,r}}

\newcommand{\calV}{\mathcal{V}}
\newcommand{\calE}{\mathcal{E}}
\newcommand{\calT}{\mathcal{T}}

\usepackage[bindingoffset=0in,left=1.3in,right=1.3in,top=1.5in,bottom=1.5in,footskip=.85in]{geometry}

\begin{document}
	\begin{abstract}We study the nilpotent part $N'$ of a pseudo-periodic automorphism $h$ of a real oriented surface with boundary $\Sigma$. We associate a quadratic form $Q$ defined on the first  homology group (relative to the boundary) of the surface $\Sigma$. Using the twist formula and techniques from mapping class group theory, we prove that the  form	$\tilde{Q}$ obtained after killing ${\ker N}$ is positive definite if all the screw numbers associated with certain orbits of annuli are positive. We also prove that the restriction of $\tilde Q$ to the absolute homology group of $\Sigma$ is even whenever the  quotient of the Nielsen-Thurston graph under the action of the automorphism is a tree. The case of monodromy automorphisms of Milnor fibers $\Sigma=F$ of germs of curves on normal surface singularities is discussed in detail, and the aforementioned results are specialized to such  situation. Moreover,  the form $\tilde{Q}$ is computable  in terms of the dual resolution  or  semistable reduction graph, as illustrated with several examples. Numerical invariants associated with $\tilde{Q}$  are able to distinguish plane curve singularities with different topological types but same spectral pairs. Finally, we discuss a generic linear  germ defined on a superisolated surface. {In this case the plumbing graph is not a tree and the restriction of $\tilde Q$ to the absolute monodromy of $\Sigma=F$ is not even}.
	\end{abstract}

	\maketitle

	\section*{Introduction}

We study the nilpotent part of pseudo-periodic automorphisms of real oriented surfaces with boundary. The monodromy of families of algebraic curves and the geometric monodromy of hypersurfaces on germs of normal surface singularities are examples of such automorphims. Our motivation comes indeed from the latter case.  Note also that the study of automorphisms of surfaces has appeared on several recent works on arithmetic and tropical geometry, see for instance \cite{Li2020surface} and \cite{corey2020ceresa}.

We associate a quadratic form with a pseudo-periodic automorphism $h$ of a real oriented surface $\Sigma$ with boundary, i.e., $\partial \Sigma \ne \emptyset$.	Let $e \in \mathbb{N}$ be the least common multiple of the orders of $h$ restricted to each periodic piece in its Nielsen-Thurston decomposition. We consider the $\ZZ$-linear 
operators
\[
N: H_1(\Sigma,\partial \Sigma; \ZZ) \longrightarrow H_1(\Sigma; \ZZ)\text{ and }
N': H_1(\Sigma; \ZZ) \longrightarrow H_1(\Sigma; \ZZ) ,
\]
given by $[\gamma] \mapsto {[h^e(\gamma) -\gamma]}$,
and we associate with them the symmetric bilinear form
\[
Q:H_1(\Sigma,\partial \Sigma; \ZZ) \times H_1(\Sigma, \partial \Sigma; \ZZ) \longrightarrow \ZZ,\hskip 5mm \hbox{ defined  by } \hskip 5mm Q(\alpha , \beta):=
	\left\langle N\alpha , \beta \right\rangle,
\]
and also 
\[
\tilde{Q}:\frac{H_1(\Sigma, \partial \Sigma; \ZZ)}{\ker N}   \times 
	\frac{H_1(\Sigma, \partial \Sigma; \ZZ)}{\ker N} \longrightarrow {\ZZ},
\]
	where 
		$\left\langle\ , \ \right\rangle:H_1(\Sigma;  \ZZ) \times H_1(\Sigma, \partial \Sigma; \ZZ) \longrightarrow \ZZ$
	denotes the usual intersection product. In a similar way we can define a symmetric bilinear form $Q'$ on $H_1(\Sigma; \ZZ)$.

The $\ZZ$-linear operator $N$ 
and the quadratic form $Q$ are designed to recover information of the unipotent part of $h_\ast$. 
The linear operator $N$ is defined as a variation operator for $h_\ast^e$ and therefore the contributions 
to $H_1(\Sigma, \partial \Sigma, \ZZ)$ coming from the periodic pieces of Nielsen-Thurston 
decomposition of $h$ are killed by $N$. The image of $N$ on the other hand can be represented by a collection of disjoint simple closed oriented curves on $\Sigma$ with the property that each of those simple closed curves do not intersect the homology of the periodic pieces but goes through at least one annular neighborhood of a separating curve in the Nielsen-Thurston decomposition of $h$. 
Note that, at least in the singularity theory setting and according to the formula for the characteristic polynomial 
of $h_\ast$ in $H_1(\Sigma, \ZZ) / \ker N'$ (see \cref{subsec:alg_inv_mon}), the unipotent part of $h$ comes from the separating nodes, and paths joining pairs of them. 
These 
relate to the aforementioned annuli.
The quadratic form $Q$ encodes the intersection of the elements of $H_1(\Sigma, \partial \Sigma, \ZZ)$ with their images under~$N$.

The basic techniques for the study of automorphisms of real oriented surfaces are the Nielsen-Thurston classification and the mapping class group. The main references are the book of Matsumoto and Montesinos \cite{MatMont} and the book of Farb and Margalit \cite{Farb}. We use these techniques to give explicit formulas of $N$ and $\tilde Q$, and sufficient conditions for $\tilde Q$ being positive definite and being even. Definiteness of a quadratic form is an important property related to a notion of convexity and, in the algebraic setting, to the Hodge index theorem. 
	
	Let $\mathcal{C}$ be a collection of pairwise disjoint simple closed curves on $\Sigma$ determining the canonical form or decomposition of the pseudo periodic automorphism $h$, see \Cref{thm:can_form}. 	
	Let 
$v, w \in H_1(\Sigma, \partial \Sigma; \ZZ)$ and let $\gamma_v, \gamma_w$ oriented curves on $\Sigma$ representing these classes. 
For $C\in\mathcal{C}$,
let $T_C$ be a tubular neighborhood  of the curve $C$, and 
let $s_C=\scn (h^e, T_C)$ be the screw number associated to the annulus $T_C$, see \cref{def:sc}, and \cref{lem:twist_number}. 
The  main results about $N$ and $\tilde{Q}$ are \Cref{thm:def_pos}, \Cref{cor:even} and formulas \eqref{eq:formN} and \eqref{eq:formtQ}. These statements can be summarized as follows.

\begin{thmA}The 
operator $N$ and the quadratic form $\tilde Q$ are given by the formulas	
\[
N(v)= \sum_{C\in\mathcal{C}}  s_C \langle [C], v \rangle [C] \quad \hbox{ and } \quad  \tilde{Q}(v,w)= \sum_{C\in\mathcal{C}} s_C \langle [C], v \rangle \langle [C],w \rangle.
\]
The bilinear form $\tilde Q$ is positive definite if all the screw numbers associated to non-nullhomotopic curves $C$ are positive and $Q'$
is  even  whenever the quotient of the Nielsen-Thurston graph $G(h)$ under the action induced by $h$  is a tree. 
\end{thmA}

We can interpret this theorem in the following way. The space $H_1(\Sigma, \partial \Sigma; \ZZ) / \ker N$ can be identified with the homology
in degree~$1$ of the Nielsen-Thurston graph of the monodromy automorphism relative to the vertices coming from the boundary.
The screw numbers of $h^e$ define a diagonal positive definite form in the group of $1$-chains of the Nielsen-Thurston graph.
The bilinear form $\tilde Q$ is identified with the restriction of this form to the above relative homology.

The role of quadratic forms in Singularity theory has been surveyed by Wall \cite{MR1803373}, in the normal surface case, and Hertling \cite{Hertling05}. If $h$ is the geometric monodromy of a germ of hypersurface on a normal surface singularity and $\Sigma=F$ is the corresponding Milnor fiber, then the hypothesis on the screw numbers in the above theorem is satisfied, see \cref{thm:posdefsi}. The reasons behind the positivity of the screw numbers $s_C$ are the twist formula  \Cref{tf}, and the fact that the resolution data $m_i$ are always positive, see~\cref{subsec:twist-res}. Alternatively and due to a remark of Mumford  \cite[II, (b), (ii)]{Mum}, the positivity of the resolution data $m_i$ can be understood as a consequence of the negative definiteness of the intersection matrix of a resolution, see \cref{rem:neg_pos}. The hypothesis about the shape of the quotient of the Nielsen-Thurston graph $G(h)$ under the action induced by $h$ in the above theorem is satisfied for plane curve singularities, see \cref{cor:even2}. 

    Furthermore, in the case of geometric monodromies of hypersurfaces on germs of normal surfaces, there is an equivalence between  the dual graph $\gss $ of the semistable reduction, and the Nielsen-Thurston graph of $h$, see \Cref{lem:homtopy_equivalence_graphs}. In particular, there is a  map associating to a closed path $\alpha$ in $F$ its	image in the dual graph $\gss $ of the semistable reduction. This map induces isomorphisms
\[
\frac{H_1(F; \ZZ)}{W_1} \longrightarrow H_1(\gss ; \ZZ) 
	\quad \hbox{ and } \quad \frac{H_1(F, \partial F; \ZZ)}{\ker N} \longrightarrow H_1(\gss , \calD; \ZZ) \text{, see page~\pageref{calD} for }W_1\text{ and }\calD,
\]
and allows us to perform very explicit computations of the form $\tilde{Q}$ in particular examples, see \Cref{sec:ex}. In particular, we show that numerical invariants associated to $\tilde{Q}$  are able to distinguish classic examples of pairs of plane curve singularities with different topological type but same spectral pairs, due to Schrauwen, Steenbrink, and Stevens, see~\cref{ex:SSS}.
In \cite{DBM2} Du Bois and Michel gave two infinite families of reducible plane singularities (all members of both families consist of two branches)
which are not topologically equivalent but have the same Seifert form. The members of both families
have asymptotically big Milnor numbers.  Therefore their Seifert forms become asymptotically complicated too.  We  compute in~\cref{ex:DBM} the forms $\tilde{Q}$ associated with both families. It is worth mentioning that for those families $\tilde Q$ is always defined on an abelian group of rank~$4$.  Our computation show that for pairs of members of these families of singularities (with the same Seifert form) the forms $\tilde Q$ are equivalent. This is not surprising,
as pointed out by a referee (to whom we are strongly grateful), since 
$\tilde{Q}$ depends on the Seifert form, see \cite[\S10]{Wall} or~\cite[\S2.3]{MR2919697}. Nevertheless, note that 
while $\tilde{Q}$ is weaker than the Seifert form it is also computationally much simpler. 
The examples of Schrauwen, Steenbrink, and Stevens, which
have equal Seifert forms over $\RR$, have no equivalent
forms~$\tilde{Q}$, hence the real Seifert form does not determined
our bilinear form.
For the examples of Du Bois and Michel, we prove the equality in the same way
as they did, finding suitable bases where the matrices coincide.
Supplementary details (and possible verification) on the examples are provided in the link 
\url{https://github.com/enriqueartal/QuadraticFormSingularity} which can be executed using \texttt{Binder}.

Finally, a few words linking the results presented here with a future project. The  
operator $N$ and the quadratic form $\tilde Q$ studied in this paper aim to describe the topological  part of the multiplication by 
a germ  of holomorphic map $f:(\CC^{n+1},0)\to(\CC,0)$  and of the residue pairing
defined on the Jacobian module of~$f$. Notice that the multiplication by $f$ map and the residue pairing 
encode analytic information in contrast to $N$ and $\tilde Q$ which encode just topological information.

The multiplication by $f$ map on the jacobian module $\Omega_f$ is given by $[gdz] \mapsto [(fg)dz]$. It is a nilpotent map with index $\leq n+1$, and it is trivial if and only if $f$ is right-equivalent to a quasihomogeneous polynomial.   Varchenko established that the nilpotent operator of the homological monodromy, and the map $Gr_\mathcal{V} \{f\}$, the graded multiplication by $f$ map with respect to the Kashiwara-Malgrange filtration $\mathcal{V}$, have the same Jordan block structure \cite{Varch81}.


	
	
	
Using the Grothendieck local duality theorem, one can define a  nondegenerate symmetric {\em residue pairing} $\res_{f,0} : \Omega_f \times \Omega_f \rightarrow \C$ as
\[
\res_{f,0}([g_1 dz],[g_2dz]) = \left(\frac{1}{2 \pi i}\right)^{n+1} \int_{\Gamma_\epsilon} \left ( \frac{g_1(z) g_2(z)}{ \prod_{j=0}^n f_j} \right ),
\]
where $\Gamma_\epsilon$ is a ($n+1$)-real vanishing cycle determined by the partial derivatives $\partial f / \partial z_j$,  
and define the bilinear form	$\res_{f,0}(\{f\} \bullet, \bullet)$, which degenerates on $\ker \{f\}$.
	
It is worthwhile to notice that, motivated by previous results relating the signatures  of the residue pairing $\res_{f,0}(\{f\} \bullet, \bullet)$ to indices of vector fields \cite{GGMM}, the fourth named author initiated a program to analyze additive expansions of the bilinear forms $\res_{f,0}(\{f\}^l \bullet, \bullet)$ and to compare them to some topological bilinear form introduced by Hertling \cite{Hertling99, Hertling02},
see for instance the results from~\cite{Dela-Rosa}.

\medskip



	\section{Nielsen-Thurston theory}
	
	Let $\Sigma$ be a real oriented surface with $\partial \Sigma \ne \emptyset$. Assume  that $\Sigma$ has genus $g$ and $r$ boundary components. Let $\modgr (\Sigma, \partial \Sigma)$ denote the mapping class group of surface diffeomorphisms of the surface of $\Sigma$ that restrict to the identity on the boundary components and up to isotopy fixing the boundary.  
	
	\begin{example}  Let $\mathcal{D} : \mathbb{S}^1 \times [-\frac{1}{2},\frac{1}{2}] \rightarrow \mathbb{S}^1 \times [-\frac{1}{2},\frac{1}{2}]$ be the homeomorphism defined by $\mathcal{D}(x,t)=(x+t,t)$, where $\mathbb{S}^1$ is identified with $\mathbb{R} / \mathbb{Z}$.  Let $\mathcal{A}$ be an annulus. A homeomorphism $h: \mathcal{A} \rightarrow \mathcal{A}$ is called a \emph{right-handed Dehn twist} if there exists a parametrization $\eta : \mathbb{S}^1 \times [-\frac{1}{2},\frac{1}{2}] \rightarrow \mathcal{A}$ such that $h = \eta \circ \mathcal{D} \circ \eta^{-1}$.  The mapping class group $\mathrm{Mod}_{0,2} (\mathcal{A}, \partial \mathcal{A})$ is isomorphic to $\mathbb{Z}$, and it is generated by a right-handed Dehn twist.
	\end{example}
	
	The Nielsen-Thurston classification of mapping classes \cite[Theorem 13.2, see also the statement in page $11$]{Farb} says that for each mapping class one of the following exclusive statements is satisfied.	
	\begin{enumerate}
		\item $h$ is periodic, i.e., there exists $n \in \NN$ such that $h^n=  \id \in  \modgr (\Sigma, \partial \Sigma)$.
		\item $h$ is pseudo-Anosov. (The appropiate definition of this notion takes some time and it won't be used in the present work. We refer the interested reader to \cite{Farb} for more on this topic.)
		\item $h$ is reducible, i.e., there exists a representative $\phi$ of $h$ and a finite union of disjoint simple closed curves that is invariant by $\phi$.
	\end{enumerate}
	
	 It follows that one can cut up the surface $\Sigma$ along a collection of invariant curves into (maybe disconnected) surfaces such that the restriction of an appropriate representative of $h$ to each of these pieces is either periodic or pseudo-Anosov. When only periodic pieces appear in this decomposition, we say that $h$ is a {\em pseudo periodic mapping class}. In this work we only deal with this type of mapping classes.
	 
	 According to Nielsen-Thurston theory each pseudo periodic mapping class has a nice representative that we call {\em canonical form}. This homeomorphism is defined by the following theorem.
	
	\begin{thm}[Canonical form, {\cite[Corollary 13.3]{Farb}}] \label{thm:can_form}
		Let $h \in \modgr (\Sigma, \partial \Sigma)$ be pseudo periodic. Then there exists a collection~$\calC$ of pairwise disjoint simple closed curves on $\Sigma$, including curves parallel to all boundary components, a collection $\calT$ of pairwise disjoint tubular neighbourhoods $\calT_i$ of each $C_i$ in $\calC$, and a representative $\phi: \Sigma \to \Sigma$ of $h$ in $\modgr(\Sigma, \partial \Sigma)$ such that 
		
			\begin{enumerate}[label=\rm(\roman{enumi})]
			\item \label{ref:i}The multicurve $\calC$ and the multiannulus $\calT$ are invariant by $\phi$, that is, $\phi(\calC) = \calC$, and $\phi(\calT)=\calT$.
			\item \label{ref:ii}The automorphism $\phi$ restricted to suitable unions of components of the closure of $\Sigma \setminus \calT$ is periodic. 
			\item \label{ref:iii}If $n$ is a common multiple of the periods of the components of $\Sigma \setminus \calT$, then $\phi^n$ is the composition of non-trivial (and non necessarily positive) powers of right-handed Dehn-twist along all the curves in $\calC$.
		\end{enumerate}
		Denote by $\calC^+$ the subcollection of curves of $\calC$ that are not parallel to the boundary, which are called \emph{separating curves}. We will always assume that  $\calC^+$ is minimal. 
	\end{thm}

 In the sequel we may identify $h$ and $\phi$ if no confusion is likely.

\begin{remark}  A.~Pichon gives a characterization of  pseudo periodic automorphisms corresponding to the  geometric monodromy of the Milnor fibration of a germ of a 
hypersurface on a normal surface  as those such that all the powers in \ref{ref:iii} above are {positive} \cite{Pich}. We will consider such situation in \Cref{sec:applst} and \Cref{sec:ex}. \end{remark}

	The behavior of the automorphism $h$ in the collection  $\calT$ of annuli is described by a rational amount of rotation.  The notion of  \emph{screw numbers}, that we define next, measures this amount. 
	
The action of $h$ and its powers partitions the collection $\calC$ into orbits of curves. Let $\calC_\ast:=\{C_1, \ldots, C_d\} \subset \calC$ be an orbit of curves defined by $h$. Let $\delta \in \{d, 2d\}$, where $\delta=d$ if $h^d$ sends $C_1$ to $C_1$ with the same orientation and $2d$ otherwise. Let $m_1, m_2 \in \NN$ be the periods of $h$ restricted to the periodic pieces on each side of the orbit and let $n:= \lcm(m_1,m_2)$. Then, we have that $h^{n}$ restricted to 
$\calT_1$ is an integral power, let us say $\ell$, of a right-handed Dehn twist around $C_1$. 
	
	\begin{definition}\label{def:sc} 
	With the above notations, we define the \emph{screw number} of $h$ at the orbit $\calC_\ast$ by 
	\[
	\scn(h,\calC_\ast)=\scn(h,\calT_\ast):=\frac{\delta \ell}{n}.
	\]
	\end{definition}
	
	
	\begin{remark}\label{rem:screw_number}
		The screw number is defined for an orbit (by $h$) of simple closed curves. 
		This is because the quantity does not depend on the curve chosen to compute it. In the same way, 
		when it is more convenient, we speak of the screw number at an orbit of annuli or at a given annulus 
		where these are taken to be tubular neighborhoods of an orbit of simple closed curves.
	\end{remark}

	We recall now the notion of {intersection product} 
	\begin{equation}\label{eq:intersection_product}
\langle\cdot, \cdot\rangle:H_1( \Sigma; \ZZ) \times H_1(\Sigma, \partial \Sigma; \ZZ)  \to \ZZ
\end{equation}
on $\Sigma$ that is computed as follows.  Firstly, every  element $\alpha \in H_1(\Sigma;\ZZ)$ can be represented by a disjoint union $\gamma_\alpha$ of simple closed curves in $\Sigma$
and any primitive element 
$\beta\in H_1(\Sigma, \partial \Sigma;\ZZ)$ can be represented by a disjoint union of simple closed curves and properly embedded arcs that we denote by $\gamma_\beta$. 
Now, for two such  elements $\alpha, \beta$ we can actually  define 
	$\langle\alpha, \beta\rangle$ to be the algebraic intersection number of $\gamma_\alpha$ and $\gamma_\beta$, i.e., \begin{equation}
	\langle\alpha, \beta\rangle := i(\gamma_\alpha, \gamma_\beta).
	\end{equation} 
Finally, one can compute this number by taking representatives of  $\gamma_\alpha$ and $\gamma_\beta$ in their isotopy classes such that they intersect transversely and then one counts $+1's$ and $-1's$ on each intersection point according to the orientation of $\Sigma$. Note that in this way the intersection form is well defined, non degenerate, and that swapping the factors in the domain of the intersection form results in the multiplication by $-1$ on its value.

	The next lemma explains the geometric meaning of the screw numbers.

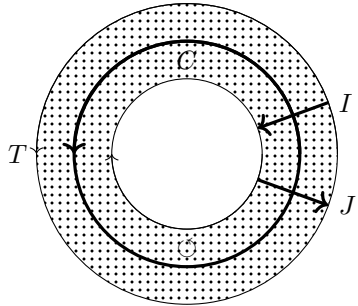
\begin{figure}[ht]
\centering
\begin{tikzpicture}
\pgfdeclarepatternformonly{my dots}{\pgfqpoint{-1pt}{-1pt}}{\pgfqpoint{5pt}{5pt}}{\pgfqpoint{6pt}{6pt}}%
{
    \pgfpathcircle{\pgfqpoint{0pt}{0pt}}{.5pt}
    \pgfpathcircle{\pgfqpoint{6pt}{6pt}}{.5pt}
    \pgfusepath{fill}
}
\draw (0,0) circle [radius=1cm];
\draw[->] (1,0) arc [start angle=0,end angle=-180,radius=1cm];
\draw (0,0) circle [radius=2cm];
\draw[->] (2,0) arc [start angle=0,end angle=180,radius=2cm];
\fill[pattern=dots] (2,0)arc [start angle=0,end angle=360,radius=2cm] 
(1,0) arc [start angle=0,end angle=-360,radius=1cm];
\draw[line width=1.2] (0,0) circle [radius=1.5cm];
\draw[->, line width=1.2] (1.5,0) arc [start angle=0,end angle=180,radius=1.5cm];
\node[below] at (0,1.5) {$C$};
\draw[line width=1.2,->] (20:2) node[right] {$I$}--(20:1);
\draw[line width=1.2,->] (-20:1)  -- (-20:2)node[right] {$J$};
\node[left] at (-2,0) {$T$};
\node at (0,-1.25) {$\circlearrowleft$};
\end{tikzpicture}
\caption{The annulus is oriented counterclockwise and induces the orientation on the boundary.}
\label{(fig:lema_screw)}
\end{figure}

\begin{lemma}\label{lem:twist_number}
		Let $h$ be a pseudo-periodic automorphism $h : \Sigma \rightarrow \Sigma$. Let $\calT_\ast$ be an orbit of $d$ annuli by the action of~$h$, and let $T$
		be an annulus in $\calT_\ast$. 

Let $e\in \NN$ be such that $h^e|_{\partial T} =\id$. Let $I, J \hookrightarrow T$ 
		be two properly embedded, disjoint and oriented arcs with one end on each boundary component of $\partial T$ and let 
		$C$ be the core curve of $T$ suitably oriented.
		
We have the following: 
\begin{enumerate}[label=\rm(\roman*)]
\item \label{lem:twist_number_i} $\scn(h^e, T) \in \NN$;
\item \label{lem:twist_number_ia} $\scn(h^e, T)= \frac{e}{d} \cdot \scn(h,  T)$;
\item \label{lem:twist_number_ib}$h^e
(I)-I=\scn(h^e, T)\cdot C \in H_1(F, \ZZ)$.
			
			\item \label{lem:twist_number_ii} 
			$\langle h^e(I)-I, J\rangle = \pm\frac{e}{d} \scn(h, T)$ if $J\sim\pm I$;
			
			\item \label{lem:twist_number_iia} in particular, $\langle h^e(I)-I, I\rangle = \frac{e}{d} \scn(h, T)$.
		\end{enumerate}
	\end{lemma}

	\begin{proof}
		The statement \cref{lem:twist_number_i} 
		follows from the hypothesis that $h^e|_{\partial T} =\id$ and the definition of screw number. The statement \cref{lem:twist_number_ia}  holds because of the additivity of screw numbers in automorphisms of cylinders.
		$\scn(h^e, T)= \frac{e}{d} \cdot \scn(h,T)$ \cref{lem:twist_number_ia}. 
		Observe also that $d$ divides $e$.
		To state \cref{lem:twist_number_ib}, note first  that $h^e|_{T}(I)-I$ is a well defined element in the absolute homology of $T$ because $h^e|_{T}(I)$ and $I$ share their ends, and have opposite orientation at their ends. Therefore, $h^e|_{T}(I)-I$ has to be an integral multiple of $C$. This integer equals the number of Dehn twists that  $h^e|_{T}$ consists of. By \cref{lem:twist_number_ia}, it is exactly $\scn(h^e, T)= \frac{e}{d} \cdot \scn(h,  T)$.
				To prove the two last items, we observe the following. We consider a model  $\MS^1 \times [-\frac{1}{2},\frac{1}{2}]$ for our annulus $T$ and without loss of generality we may assume that $I$ is identified with the vertical segment $\{1\} \times [-\frac{1}{2},\frac{1}{2}]$ oriented from $-\frac{1}{2}$ to $\frac{1}{2}$. So $h^e|_{T}(I)-I$ may be represented by $\frac{e}{d}$ circles properly oriented.
	\end{proof}

	Finally, it is useful to associate to a pseudo-periodic surface automorphism as above, its Nielsen-Thurston graph (sometimes also called partition graph), see for example \cite[Section 6.1]{MatMont} for more on this graph.
		
		\begin{definition}\label{def:nielsen_graph}
		Let $h: \Sigma \to \Sigma$ be a pseudo-periodic automorphism $h: \Sigma \rightarrow \Sigma$ in canonical form (see \cref{thm:can_form})
		and let $\calC$ be the collection of separating curves. We define the {\em  
			Nielsen-Thurston graph}  associated with $h$, and denote by $G(h)$, as the graph that
		\begin{enumerate}[label=\rm(\roman{enumi})]
			\item has one vertex $v$ for each connected component $\Sigma_v$ in $\Sigma 
			\setminus \calC$
			\item one edge connecting the vertices 
			$v$ and $w$ for each separating curve $\calC_i$  such that the boundary components of
			$\calT_i$ are contained in $\Sigma_v$ and $\Sigma_w$;
note that $v$ might be equal to $w$. 
		\end{enumerate}
		Moreover, there is a collapsing map $\xi: \Sigma \rightarrow G(h)$, collapsing  each connected component $\Sigma_v$ to the vertex $v$ and projecting to the factor $[-\frac{1}{2},\frac{1}{2}]$ along the annuli, and an induced periodic isomorphism $h_{G(h)}$ such that $\xi \circ h = h_{G(h)} \circ \xi$.
	\end{definition}

\begin{remark}\label{re:inducv}The topology of the Nielsen-Thurston graph $G(h)$ encode important topological features of $h$ \cite[Section 7.4]{Weber}. \end{remark}

\section{The 
operator \texorpdfstring{$N$}{N} and the quadratic form \texorpdfstring{$\tilde{Q}$}{tQ}}
\mbox{}

\subsection{The 
operator \texorpdfstring{$N$}{N}}
	\mbox{}
	
	Let $h: \Sigma \to \Sigma$ be a pseudo-periodic automorphism of a surface with 
	$h|_{\partial \Sigma}= \id$. 	Let 
\[
h_\ast: H_1(\Sigma, \partial \Sigma) \to H_1(\Sigma, \partial \Sigma)
\] 
be the induced operator on the 
first relative homology group of $(\Sigma, \partial \Sigma)$ with $\ZZ, \QQ$ or $\CC$ coefficients. 	

For the linear operator $h_\ast$ defined on the vector space $H_1(\Sigma, \partial \Sigma; \QQ)$  or $H_1(\Sigma, \partial \Sigma; \CC)$ there is a canonical 
decomposition 
\[
h_\ast= h_s h_u = h_uh_s,
\]
where $h_u$ is the unipotent part 
and $h_s$ is the semisimple part of $h_\ast$. The semisimple part  
$h_s$ codifies the information about the eigenvalues of $h_\ast$ and the unipotent $h_u$ 
	codifies the information about the Jordan blocks of $h_\ast$. In particular,  $h_s$ 
	is a diagonalizable linear operator and $h_u$ is, up to change of basis, an 
	upper triangular matrix with 1's on the diagonal.
	
	 In the literature,  one can find  three, a priori, different nilpotent operators associated with $h_\ast$:
	\begin{equation}\label{eq:nilpotent}
	h_u - \id, \qquad -{\rm log} h_u, \qquad (h_\ast^e-\id)/e,
	\end{equation}
	 defined only on $\QQ$ or $\CC$ and the exponent  $e \in \NN$ being the smallest natural number such that $h^e$ is the identity restricted to each periodic piece, i.e., $e$ is the least common multiple of the orders of the periodic pieces\label{pag:e}. However, we define another nilpotent operator that will play a central role in the present work. 
	
	

	\begin{definition}\label{def:nilpotent_N}
		Let $h$ be a pseudo-periodic automorphism of a compact surface with boundary $\Sigma$ and let $e \in \NN$ be the smallest natural number such that $h^e$ is the identity restricted to each periodic piece (i.e. $e$ is the least common multiple of the orders of the periodic pieces). Then we define the \emph{nilpotent-like operators}
		$N: H_1(\Sigma,\partial \Sigma; \ZZ) \longrightarrow H_1(\Sigma; \ZZ)$  and $N': H_1(\Sigma; \ZZ) \longrightarrow H_1(\Sigma; \ZZ)$ given by 
\[
N([\gamma]) =   [h^e(\gamma) -\gamma]\in H_1(\Sigma; \ZZ);
\]
	$N'$ is defined by same expression.
	\end{definition}

This operator is related to the others 
in \eqref{eq:nilpotent} but has the additional property that is well defined on the module $H_1(\Sigma, \partial \Sigma;\ZZ)$. It is nilpotent because its definition coincides with the third equation definition of the nilpotent operator of  \eqref{eq:nilpotent} times a constant.


%
	
	\subsection{The quadratic forms $Q$ and $\tilde{Q}$}
\mbox{}

\begin{definition}\label{<>} We associate to the 
operator $N$ the 
		 form
		\[
		Q:=
		\langle N\cdot , \cdot\rangle : 
		H_1(\Sigma, \partial \Sigma;\ZZ) \times H_1(\Sigma, \partial \Sigma;\ZZ) 
		\longrightarrow \ZZ
		\]
		which can be pushed down to a form
\[
\tilde{Q}: \frac{H_1(\Sigma, \partial \Sigma; \ZZ)}{\ker N} \times \frac{H_1(\Sigma, \partial \Sigma; \ZZ)}{\ker 
N} \longrightarrow \ZZ.
\]
Analogously we associate to $N'$ the form
		\[
		Q':=
		\langle N'\cdot , \cdot\rangle : 
		H_1(\Sigma;\ZZ) \times H_1(\Sigma;\ZZ)
		\longrightarrow \ZZ.
		\]
	\end{definition}

	\begin{remark}
		 The  forms defined above are symmetric. Indeed, $h^e$ is a diffeomorphism of the surface so $\langle v,w\rangle =\langle h^e(v),h^e(w)\rangle $. It follows by a straightforward calculation that $\langle Nv,w\rangle =-\langle v,Nw\rangle $ and so $\langle Nv, w\rangle  = \langle Nw,v\rangle $.

	\end{remark}
	
	The bilinear form $\tilde{Q}$ can be computed in terms of the screw numbers with help of \Cref{lem:twist_number}. Let 
$v,w \in H_1(F, \partial F; \ZZ)$ and let $\gamma_v,\gamma_w$ be  oriented curves on $F$ representing the classes $v,w$. For $C\in\mathcal{C}$,  let $T_C$ be a tubular neighborhood  of the curve $C$, and  
let $s_C=\scn (h^e, T_C)$ be the screw number associated to the annulus $T_C$, see \cref{lem:twist_number}. 
The expressions 	
\[
N(v)= \sum_{C\in\mathcal{C}}  s_C \langle [C], v \rangle [C] \quad \hbox{ and } \quad  \tilde{Q}(v,w)= \sum_{C\in\mathcal{C}} s_C \langle [C], v \rangle \langle [C],w \rangle
\]
are consequence of \cref{def:nilpotent_N},  \cref{lem:twist_number}, and \cref{<>}.

\medskip

The following theorem is the main result of this paper.
	
	\begin{thm}\label{thm:def_pos}
If all the screw numbers $s_C$ associated to orbits of annuli whose core curves are non-nullhomotopic are positive, then the symmetric bilinear form $\tilde{Q}$ introduced in 
		\emph{\Cref{<>}} is positive definite.\end{thm}
	
	\begin{proof} 
	
		First we are going to prove the semi-definiteness for $Q$. 
		
		So let $v \in H_1(\Sigma, \partial \Sigma; \ZZ)$ be an element and let $\gamma_v \in \Sigma$ be a collection of disjoint oriented simple closed curves and oriented properly embedded arcs representing $v$. Let $\calC$ be the collection of separating simple closed curves of the monodromy including curves parallel to all boundary components. Without loss of generality, we can assume that $\gamma_v$ and $\calC$ intersect transversely.
		
		By definition of pseudo-periodic, a power of  the geometric monodromy, let's say $h^e$, can be taken to be the identity outside a small tubular neighborhood of $\calC$, we can assume that, after a small isotopy, all the intersection points of $h^e(\gamma_v)$ and~$\gamma_v$ occur in that small neighborhood.
		
		Let $C_1, \ldots C_k$ be the subcollection of oriented simple closed curves in $\calC$ that $\gamma_v$ intersects non-emptily  and such that the class of $C_i$ in homology is not $0$. Let $T=T_1 \cup \ldots \cup T_k$ be the union of disjoint tubular neighborhoods $T_i$ of these $k$ curves. We can assume that $\gamma_v \cap T$ is a collection $I$ of oriented disjoint segments. And for each $i$, we define $\gamma_v \cap T_i = I_i$ which is an union of  segments $I_{i,1}, \ldots, I_{i,a_i}$ each connecting one boundary component of $T_i$ with the other. By definition of $e$, $h^e|_{\partial T_i}= \id$ for all $i=1, \ldots, k$. Moreover, $h_\ast^e(I_{i,k})-I_{i,k}$ is, in the integral relative homology of $\Sigma$, an integral multiple of $C_i$. Actually, this number is by construction $s_i:=\scn(h^e, T_i)$ which can be computed using \Cref{lem:twist_number} and \Cref{tf} from the dual decorated graph.
		
		The segments in $I_{i}$ are oriented and so by fixing an order on the two boundary components of $T_i$, let us say $\partial_1 T_i$ and  $\partial_2 T_i$, we can distinguish between the union of those arcs that go from $\partial_1 T_i$ to $\partial_2 T_i$ which we denote by $I^+_i$ and the
		others which go in the opposite direction and we denote by $I^-_i$. 
		So $I_i = I^+_i\sqcup I^-_i$.
		Let $p^+_i$ be number of arcs in $I^+_i$ and analogously define $p^-_i$ as the number of arcs in $I^-_i$. Remark that $a_i=p^+_i + p^-_i$.
		
		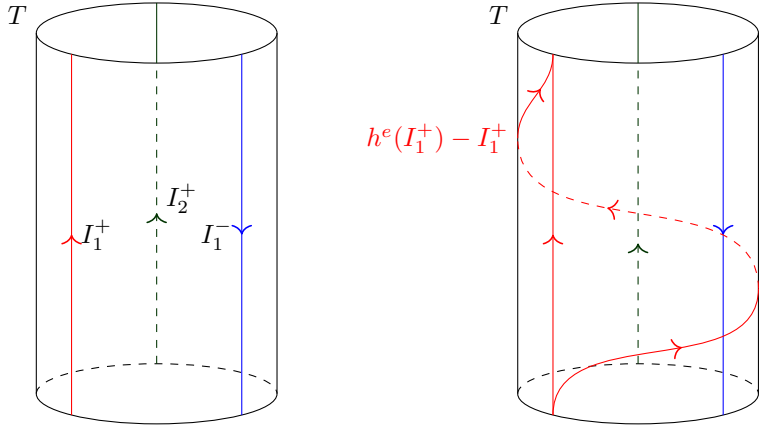
\begin{figure}[!ht]
			\begin{tikzpicture}[scale=.8]
\tikzset{->-/.style={decoration={
  markings,
  mark=at position #1 with {\arrow[scale=1.5]{>}}},postaction={decorate}}}
\draw (0,3) ellipse [x radius=2,y radius=.5];
\draw[dashed] (-2,-3) arc [start angle=180,end angle=0,x radius=2,y radius=.5];
\draw(2,-3) arc [start angle=0,end angle=-180,x radius=2,y radius=.5];
\draw (2,-3) -- (2,3);
\draw (-2,-3) -- (-2,3);
\node[above left] at (-2,3) {$T$};
\draw[blue,->-=.5] ($(0,3)+({2*cos(45)},{.5*sin(-45)})$) -- ($(0,-3)+({2*cos(45)},{.5*sin(-45)})$) node[left,black,pos=.5] {$I_1^-$};
\draw[red,->-=.5] ($(0,-3)+({2*cos(135)},{.5*sin(-135)})$) -- ($(0,3)+({2*cos(135)},{.5*sin(-135)})$) node[right,black,pos=.5] {$I_1^+$};

\draw[dashed,green!20!black,->-=.5] ($(0,-3)+({2*cos(90)},{.5*sin(90)})$) -- ($(0,3)+({2*cos(90)},{.5*sin(-90)})$) node[right,black,pos=.55] {$I_2^+$};
\draw[green!20!black] ($(0,3)+({2*cos(90)},{.5*sin(-90)})$) -- ($(0,3)+({2*cos(90)},{.5*sin(90)})$);

\begin{scope}[xshift=8cm]
\tikzset{->-/.style={decoration={
  markings,
  mark=at position #1 with {\arrow[scale=1.5]{>}}},postaction={decorate}}}
\draw (0,3) ellipse [x radius=2,y radius=.5];
\draw[dashed] (-2,-3) arc [start angle=180,end angle=0,x radius=2,y radius=.5];
\draw(2,-3) arc [start angle=0,end angle=-180,x radius=2,y radius=.5];
\draw (2,-3) -- (2,3);
\draw (-2,-3) -- (-2,3);
\node[above left] at (-2,3) {$T$};
\draw[blue,->-=.5] ($(0,3)+({2*cos(45)},{.5*sin(-45)})$) -- ($(0,-3)+({2*cos(45)},{.5*sin(-45)})$);
\draw[red,->-=.5] ($(0,-3)+({2*cos(135)},{.5*sin(-135)})$) -- ($(0,3)+({2*cos(135)},{.5*sin(-135)})$);

\draw[dashed,green!20!black,->-=.4] ($(0,-3)+({2*cos(90)},{.5*sin(90)})$) -- ($(0,3)+({2*cos(90)},{.5*sin(-90)})$) ;
\draw[green!20!black] ($(0,3)+({2*cos(90)},{.5*sin(-90)})$) -- ($(0,3)+({2*cos(90)},{.5*sin(90)})$);

\draw[red,->-=.6] ($(0,-3)+({2*cos(135)},{.5*sin(-135)})$) to[out=90,in=-90] (2,-1.25);
\draw[,dashed, red,->-=.6] (2,-1.25) to[out=90,in=-90] (-2,1.25);
\draw[red,->-=.6] (-2,1.25) node[left] {$h^e(I_1^+)-I_1^+$} to[out=90,in=-90] ($(0,3)+({2*cos(135)},{.5*sin(-135)})$);
\end{scope}
\end{tikzpicture}
			\caption{On the left we see three arcs $I^+_1,I^+_2$, and $I^-_1$. Two of them oriented in the same direction and the other one in the opposite direction. On the right we see the cycle $h^e(I^+_1)-I^+_1$ and observe that its intersection with $I^+_2$ (in that order) is $+1$, and its intersection with $I^-_1$ is~$-1$.}
			\label{fig:theorem_arcs_1}
		\end{figure}

		For $I_{i,b} \in I_i$, we have that $\langle h^e(I_{i,b})-I_{i,b}, I_i\rangle $ equals to
		\[
		\sum_{j=1}^{a_i} \langle h^e(I_{i,b})-I_{i,b}, I_{i,j}\rangle  = \sum_{j=1}^{a_i} \langle s_i C_i, I_{i,j}\rangle   = 
		\begin{cases}
		s_i(p^+_i-p^-_i), & \text{for }  I_{i,b} \subset I^+_i\\
		s_i(p^-_i-p^+_i), & \text{for }  \text{for }  I_{i,b} \subset I^-_i
		\end{cases}.
		\]

		The above formula follows because  $h^e$ is a composition of right-handed Dehn twists in a tubular  neighborhood $\calT$ of $\calC$. In particular the next to last  (resp. last) equality follows from \Cref{lem:twist_number}\ref{lem:twist_number_ii} (resp. \ref{lem:twist_number_ib}).

		With these two equalities, we can compute
		\begin{equation}\label{eq:sum_cylinder}
		\begin{split}\langle h^e(I_i)-I_i, I_i\rangle &= \sum_{j=1}^{a_i} \langle h^e(I_{i,j})- I_{i,j}, I_i\rangle  
		= \sum_{b=1}^{p^+_i}  s_i(p^+_i-p^-_i) + \sum_{c=1}^{p^-_i} s_i(p^-_i-p^+_i) \\
		&= s_ip^+_i(p^+_i-p^-_i) + s_ip^-_i(p^-_i-p^+_i) = s_i(p^+_i-p^-_i)^2 \geq 0.
		\end{split}
		\end{equation}
		Now, we have that
		\begin{align}\label{eq:Q}
		Q(v,v)=\langle Nv,v\rangle  &= \sum_{i=1}^k \sum_{j=1}^k\langle h_\ast^e(I_i)-I_i, I_j\rangle  = \sum_{i=1}^k  \langle h^e(I_i)-I_i, I_i\rangle ,\end{align}
		because $C_i$ and $I_j$ are disjoint  and $\langle h_\ast^e(I_i)-I_i, I_j\rangle =(p^+_i - p^-_i)s_i\langle C_i, I_j\rangle =0$ whenever $i \ne j$.
		
		Finally, the result that $Q$ is positive semi-definite follows because  all the summands of the last term in \eqref{eq:Q} are non-negative by \eqref{eq:sum_cylinder}.
		
		Now we need to prove that $\tilde{Q}$ is positive definite if and only if $Q$ is positive semi-definite. This is equivalent to showing that $Q(v,v)=0$ if and only if $v \in \ker N$. It is clear that if $ v \in \ker N$ then $Q(v,v)=0$. Assume now that $Q(v,v)=0$, then, by \eqref{eq:sum_cylinder} and \eqref{eq:Q}, this implies that $p_i^+ = p_i^-$ for all $i=1, \ldots, k$. This is the same as saying that in each cylinder, there are as many intervals going in one direction as intervals going in the opposite direction. In this case, we can invoke \Cref{lem:twist_number} and see that this implies that $h^e(I_i)-I_i$ consists of a number of positive multiples of $C_i$ and the same absolute number of negative multiples of $C_i$. Hence $h^e(I_i)-I_i$ is $0$ in homology and so $v \in \ker N$.
	\end{proof}

	Recall that an integral quadratic form  ${Q}$ is called even if  ${Q} (u,u)$ is always an even number.

	\begin{cor}\label{cor:even} The symmetric bilinear map $Q'$ on $H_1(F; \ZZ)$ is even if the quotient of the Nielsen-Thurston graph $G(h)$ under the action induced by $h$  is a tree. 
	\end{cor} 
	\begin{proof}The statement follows because, under the above hypothesis,  every embedded closed curve $\gamma_v$ must go through each orbit  $\calT$ of annuli an even number of times.
	\end{proof}

	To finish this section we deduce explicit  formulas  for $N$ and $\tilde{Q}$, see \eqref{eq:formN} and \eqref{eq:formtQ} below.

The following expressions for $N$ and $\tilde{Q}$ are implicit in the  proof of \Cref{thm:def_pos}.  For each tubular neighborhood $T_i$, let us fix a choice of an order of the components $\partial_1 T_i$ and $\partial_2 T_i$, and consider the numbers $s_i$, $p^+_i, p^-_i,$ and $a_i$ be as in the proof of \Cref{thm:def_pos}. Moreover, set $b_i=p^+_i - p^-_i$. Then, we have
\begin{equation}\label{eq:formN}N(v)= \sum_{i=1}^k b_i s_i C_i.\end{equation}
Moreover, let us rename  now $p^+_i$, $p^-_i$, $a_i$ and $b_i$ as  $p^+_i(v), p^-_i(v), a_i(v),$ and $b_i(v)$. 
Given an element $w \in H_1(F, \partial F; \ZZ)$, let us choose a representative $\gamma_w$ for $w$,  and,  taking into account the previous choice of an order of the components $\partial_1 T_i$ and $\partial_2 T_i$,  we define analogously the numbers $p^+_1(w)$, $p^-_i(w), a_i(w),$ and $b_i(w)$ for each $i$. Then,
\begin{gather*}\left\langle \sum_i b_i(v) s_i C_i, \gamma_w \right\rangle =  
\left\langle \sum_{j} p^+_j(v) s_j C_j - \sum_{k} p^-_k(v) s_k C_k, \gamma_w \right\rangle=\\
\sum_{j} p^+_j(v) s_j (p^+_j (w) - p^-_j (w)) - \sum_{k} p^-_k(v) s_k (p^-_k (w) - p^+_k (w)),
\end{gather*}
and, therefore, we get
\begin{equation}\label{eq:formtQ}\tilde{Q}(v,w)= \sum_i s_i b_i(v) a_i(w).\end{equation}

	\section{Applications to Singularity Theory}\label{sec:applst}

Let $f:(X,0) \to (\C,0)$ be the germ of a reduced holomorphic map germ defined on an isolated complex surface singularity $(X,0)$. In this situation there is a locally trivial fibration  
\[
\tilde{f}:=f_{|_{f^{-1}(\partial D_\delta) \cap B_\epsilon \cap X}}: f^{-1}(\partial D_\delta) \cap B_\epsilon \cap X \to \partial D_\delta,
\]
where $B_\epsilon$ is a closed ball of small radius $\epsilon>0$ and $D_\delta \subset \C$ is a small disc of radius $\delta>0$ small enough with respect to $\epsilon$. The map $\tilde{f}$ is called the {\em Milnor-L\^e fibration} and the fact that it is a fibration was proven in \cite{Le}. The fiber $\tilde{f}^{-1}(t)$ is called the Milnor fiber of $f$ at $0$. It is a compact and oriented surface with non-empty boundary. From now on, we denote it by $F$. 
	
Let $h: F \to F$ be the automorphism of $F$ defined by the locally trivial fibration $\tilde{f}$, up to isotopy. The  automorphism $h$ is called  \emph{geometric monodromy}. Since $f$ is reduced we have that $h|_{\partial F}= \id$. The geometric monodromy is known to be a pseudo-periodic automorphism of $F$ \cite[Theorem 4.2, its addendum and Section 13]{en}.	 Morevover, the geometric monodromies induced by such map germs were characterized by  Pichon in \cite{Pich} as those pseudo-periodic surface homeomorphisms $h$ such that for some $e \in \NN$, $h^e$ is a composition of positive powers of right-handed Dehn twists around disjoint simple closed curves including all boundary components. 

Therefore, the results of the previous section apply to the geometric monodromy $h: F \rightarrow F$, and from now on we set $\Sigma=F$ and $\sigma=h$. 

The rest of the paper is devoted to study the particular case of the geometric monodromy $h: F \rightarrow F$ using the additional structure coming from algebraic nature of $f : (X,0) \rightarrow (\C,0)$. 

More concretely, we will show that, as a consequence of the negative definiteness of the intersection matrix, all the screw numbers are positive. Hence, the hypothesis of \Cref{thm:def_pos} is always fulfilled.  Moreover, the hypothesis of \Cref{cor:even} is always satisfied in the case $(X,0)=(\C^2,0)$. Moreover, we will explain how to perform concrete calculations showing that numerical invariants associated to $\tilde{Q}$  are able to distinguish classic examples of pairs of plane curve singularities with different topological type but same spectral pairs, due to Schrauwen, Steenbrink and Stevens~\cite{Steen}, or same Seifert forms, due to Du Bois and Michel~\cite{DBM2}, see respectively \cref{ex:SSS} and \cref{ex:DBM}.

	\subsection{Embedded resolution, and the twist formula}\label{subsec:twist-res}
	\mbox{}
	
	
	We recall the notion of embedded resolution of the germ $f:(X,0) \to (\C,0)$  of a reduced holomorphic map germ defined on an isolated complex surface singularity $(X,0)$. 
	
	\begin{definition} An \emph{embedded resolution} of $f:(X,0) \to (\C,0)$ is a map $\pi : (\tilde{X}, E) \rightarrow (X,0)$ such that $\tilde{X}$ is smooth, $\pi$ is bimeromorphic, and the exceptional divisor $E$,  and the divisor $(f \circ \pi)^{-1}(0)$, called the total transform of $f$,  have normal crossing support on $\tilde{X}$. 
	 \end{definition}
	
It is customary to associate a {\em dual graph} with a normal crossing divisor in the following way. First, one associates a vertex $v$ with each irreducible component $E_v$ of the normal crossing divisor. {The set of vertices is denoted by $\calV$.}  Then, one associates an edge $e$ between two (not necessarily different) vertices if the corresponding irreducible components intersect (or have autointersections). {The set of edges is called $\calE$. The \emph{valency} $\delta_v$ of a vertex~$v$ is the number of adjacent edges to~$v$.}

We denote by   $\Gamma_E$ (resp.  $\Gamma_f$) the dual graph of the exceptional divisor $E$ (resp.  of the total transform of $f$).

{Different numerical decorations are attached to the vertices of  $\Gamma_E$ (or  $\Gamma_f$) for different purposes. For instance, let $g_v$ (resp. $e_v$) denote the genus of the irreducible component $E_v$ (resp. the opposite of the autointersection number $E_v^2$). Moreover, let  $\chi_v$ denote the number $2 - 2g_v -\delta_v$.}
	
The total transform of $f$ is a non reduced normal crossing divisor. The multiplicity $m_v$ of an irreducible component $E_v$ of $(f \circ \pi)^{-1}(0)$ is given by the order of vanishing of the pullback $\pi^\ast f$ of $f$ along $E_v$. Therefore the multiplicity $m_v$ is a positive integer. The collection of multiplicities $m_v$ is called the {\em numerical data} of the resolution of $f$, {and are sometimes used as decorations of the vertices of $\G_f$ too.}


	\begin{remark}\label{rem:neg_pos}According to Mumford \cite[II, (b), (ii)]{Mum} the positivity of the multiplicities can be alternatively derived from the negative definitiveness of the intersection matrix associated with the exceptional divisor $E$ of the embedded resolution $\pi$, due to Grauert, 	and the systems of linear equations $m_v \cdot E_v^2 + \sum_w m_w \cdot (E_v \cdot E_w)=0$, with unknowns the multiplicities $m_v$.
\end{remark}

	
	
	Now we  recall the twist formula that allows to 
	compute the screw numbers  $\scn(h,\calC_\ast)$, introduced in \Cref{def:sc},  in terms of the decorated dual resolution graph $\G_f$. 
	
	The twist formula 
	comes from the analysis of the monodromy of a monomial {$w_i^{m_i}w_j^{m_j}$} and 
	has appeared several times in the literature with different forms.
	The original reference is \cite[Section 2]{N} but it is neither widely known, nor easy to find. Alternative or more recent references are \cite[Proposition 3.1]{DBM} or \cite[Lemma 10.3.7]{Wall}. Even if the original setting corresponds to germs of plane curves, since the statements are local in a resolution,
	they apply for our more general setting.
	
	
\begin{remark}\label{rem:milnor_fiber}
Let us briefly recall the construction of the Milnor fibre from the dual resolution graph $\G_f$. 

\begin{itemize}
\item For each vertex $v_i$, decorated with a multiplicity $m_i$, of the dual resolution graph $\G_f$ we consider a cyclic $m_i$-fold 
covering $F_i$ of the open subset $E^\circ_i := E_i \setminus \bigcup_{j \ne v} 
E_j$ of the divisor $E_i$; the topological type of this cyclic cover depends on the function~$f$ but sometimes (e.g. when $X$ is smooth) can be determined with combinatorial data.

\item  For each edge $e$ with end vertices $v_i$ and $v_j$ 
a cylindrical piece $c_e \times I \subset \mathbb{S}^1 \times \mathbb{S}^1 \times I$ given by 
$w_i^{m_i}w_j^{m_j}=1$ for suitable analytic coordinates at the point of $E_i\cap E_j$
associated to the edge $e$. Notice that $c_e$ consists of the disjoint union of 
$m_e:= {\rm gcd} (m_i, m_j)$ curves $C_r$, each one isomorphic to $\mathbb{S}^1$. 

\item A topological model of the Milnor fiber can be obtained gluing the pieces $E^\circ_i$ and $c_e \times I$ according to the adjacencies of $\Gamma_f$ and by means of plumbing operations. 
\end{itemize}
\end{remark}


	
{A tubular neighborhood of an invariant orbit $\calC$ is a set of annuli, and, according to the construction above,}  correspond to a bamboo with $k+1$ vertices in the dual resolution graph $\G_f$ (A \emph{bamboo} is a maximal linear subgraph). 
The bamboo starts at a vertex with 
$\chi_v<0$ and ends at another such vertex; vertices~$v$ for which $\chi_v<0$ are called~\emph{nodes}. The rest of the vertices have valency $2$, {and genus 0, i.e., $\chi_v=0$}. Let $m_0, \ldots, m_k$ be the multiplicities associated to the monodromy on each of the vertices of the bamboo and let $d:=\gcd(m_0, m_1)=\gcd(m_i, m_{i+1})$. 

	\begin{prop}\label{tf}  The screw number  of the pseudo-periodic automorphism $h$ at an orbit  $\calC_\ast$ of curves is 
		\begin{equation}\label{twistf} \scn(h,\calC_\ast):=d^2 \sum_{i=0}^{k-1}\frac{1}{m_i m_{i+1}}.
		\end{equation}
		Observe that $d$ divides $\lcm(m_0, m_k)$ and so it divides any 
		multiple $e$ of $\lcm(m_0, m_k)$.
	\end{prop}
	
	\begin{proof}
		As stated, the proposition is proved in \cite[Lemma 4.24]{Por} but equivalent formulae are classic and numerous in the literature see \cite[Section 2]{N}, \cite[Proposition 3.1]{DBM} or \cite[Lemma 10.3.7]{Wall}.
	\end{proof}

	Now, we can state the following result that is a particular case of \Cref{thm:def_pos}.

\begin{thm}\label{thm:posdefsi}
The form $\tilde{Q}$ associated with the geometric monodromy of a germ $f: (X,0) \rightarrow (\C, 0)$ is positive definite.  
\end{thm}
\begin{proof}The statement  follows from    the fact that the multiplicities $m_i$ are positive, the expression for the twist formula in  \Cref{tf}, and  \Cref{thm:def_pos}, see also \Cref{rem:neg_pos} to relate this result to the negative definiteness of the intersection matrix of the exceptional divisor $E$ of the embedded resolution $\pi$.\end{proof}

\subsection{Semistable reduction}
	\mbox{}

Next we recall the  notion of  semistable reduction of the germ $f:(X,0) \to (\C,0)$  of a reduced holomorphic map germ defined on an isolated complex surface singularity $(X,0)$. In contrast with an embedded resolution,  the semistable reduction has the advantage of associating with $f$ a reduced normal crossing divisor. Let us fix as $e$ the least common divisor of the multiplicities $m_v$ of the nodes; it coincides with the value~$e$ defined in page~\pageref{pag:e}.

\begin{definition} Let $f': (f\circ \pi)^{-1}(D_\delta^*)  \to D_\delta^*$ be the restriction of $f \circ \pi$ to $(f \circ \pi )^{-1}(D_\delta^*)$ and let $\sigma: D_{\delta^{1/e}}^* \to D_\delta^*$ be the base change map given by $t \mapsto t^e$. Consider the fibered product of the $X^{(e)} :=  (f\circ \pi)^{-1}(D_\delta^*) \times_{D_{\delta}^*} D_{\delta^{1/e}}^*$ maps $f'$ and $\sigma$, and the normalization $\tilde{X}^{(e)}$ of $X^{(e)}$ which is a $V$-manifold. The natural map $f^{(e)} : \tilde{X}{(e)} \to D_{\delta^{1/e}}^*$ is called the \emph{semistable reduction} of $f$. 
\end{definition}

	\begin{figure}[ht]
\begin{tikzcd}[/tikz/column 3/.append style={anchor=base west}]
\tilde{X}^{(e)}\ar[r]&{X}^{(e)}\ar[r]\ar[d]&(f\circ \pi)^{-1}(D_\delta^*)\ar[d,"f'", start anchor=-5.629cm]\\
&D_{\delta^{1/e}}^*\ar[r, "\sigma"] & D_\delta^*\\[-20pt]
\end{tikzcd}
\caption{Construction of semistable reduction of $f$}
\end{figure}
		
The  preimage of $(f^{(e)})^{-1}(0)$  is a reduced divisor with $\mathbb{Q}$-normal crossings in $\tilde{X}^{(e)}$, see~\cite{jorge-Semistable}. We denote by $\gss$ the dual  graph of $(f^{(e)})^{-1}(0)$ starting from a minimal $\mathbb{Q}$-resolution. 

\begin{remark} For each node $v$ in the resolution graph $\G$, the semistable reduction graph $\gss$ has as many vertices as connected components of $F \cap E_v$, that is, as many vertices as pieces of the Milnor fiber lie in the corresponding circle bundle. In fact, we can mimic the construction of the Milnor fiber of \Cref{rem:milnor_fiber} but,  in this case, no cyclic cover is taken into account.
\end{remark}

Alternatively, following J.~Mart{\'i}n-Morales, see \cite{jorge-Semistable}, one can substitute the embedded resolution $\pi$ with a so-called $\mathbb{Q}$-resolution (i.e. the total space admits abelian quotient singularities and the divisor has $\mathbb{Q}$-normal crossings) for which the semistable reduction
is simpler. We use this modification in  \cref{sec:ex}. 

	\subsection{Graph manifolds}
	\mbox{}

The dual graphs $\G_E, \G_f$, and $\gss$ introduced in the previous subsection are examples of {\rm plumbing graphs}. Next we recall the notions of plumbing graphs and graph manifolds.
	 
	\begin{definition}
	A  {\em plumbing graph} $\G$ consists of the following  data.
	
	\begin{enumerate}[label=\rm(\roman{enumi})]
		\item A set of vertices $\calV$.
		\item A set of arrowheads $\calD$.
		\item A set of edges $\calE$ connecting vertices with vertices or with arrowheads. Each edge connects two vertices which are allowed to be the same, or it connects a vertex and an arrowhead. An arrowhead is connected to exactly one vertex and there is no restriction on the number of vertices and arrowheads a given vertex is connected to.
		\item For each $v \in \calV$ an ordered couple of integers $(e_v,g_v)$ where $g$ is nonnegative.
	\end{enumerate}
	\end{definition}
	
	From the above piece of combinatorial data we can construct a $3$ manifold as follows.
	
	\begin{enumerate}[label=\alph*)]
		\item For each vertex $v \in \calV$, let $E_v$ be the circle bundle with Euler number $e_v$ over the closed surface of genus $g_v$.
		\item For each edge connecting two vertices $v,u \in \calV$ do the following: Pick a trivializing open disk on the base of each circle bundle $E_v$ and $E_u$. Remove the open solid torus lying over each of the disks. And finally, glue the two boundary tori identifying base circles of one with fiber circles of the other (and viceversa).
		\item For each edge connecting a vertex $v\in \calV$ and an arrowhead $d\in\calD$ 
		pick a trivializing open disk on the base of $E_v$ and remove the open solid torus lying over the disk.%
	\end{enumerate}

	\begin{definition}
		We say that a $3$-dimensional manifold $M$ is a \emph{graph manifold} if it is diffeomorphic to a $3$-manifold constructed as above. We always assume that $M$ is oriented. 
	\end{definition}

\begin{example}\mbox{}

\begin{itemize}
\item The 3-dimensional manifold $M:=f^{-1}(\partial D_\delta) \cap B_\epsilon \cap X $ is a graph manifold corresponding to the plumbing graph $\G_E$. 
Notice that $\G_E$ graph has no arrows. 

\item The 3-dimensional manifold consisting of {the complement of} an open regular neighborhood of the link of $f$ in $M$ is a graph manifold corresponding to the plumbing graph $\G_f$. Notice that arrows of $f$ correspond to the connected components of the strict transform of $f$. Notice that in this case the numerical data, also called a {\em system of multiplicities},  can be recovered in the following topological way. A fiber of $E_v$ is isotopic to the boundary of a germ of curve $\calC_v$ in $X$. The multiplicity $m_v \in \ZZ$ associated to $f$ is the oriented intersection number of $f^{-1}(t)$ with $\calC_v$. 
 \end{itemize} \end{example}
 
 The following result is a reformulation of \Cref{cor:even} in terms of the plumbing graph. 
 
 \begin{cor}\label{cor:even2}The form $\tilde Q$ restricted to $H_1(F, \mathbb{Z})/ \ker N'$ is even, if the plumbing graph $\Gamma_f$ is a tree. In particular,  $\tilde Q$ restricted to $H_1(F, \mathbb{Z})/ \ker N'$ is even in the case of plane curve singularities $f:(\C^2,0) \to (\C,0)$.  
 \end{cor}
 \begin{proof}Under the hypothesis of \Cref{cor:even} the plumbing graph $\Gamma_f$ is a tree. This is always the case of plane curve singularities \end{proof}
 
 \begin{remark}\Cref{cor:even} and \Cref{cor:even2} are to be compared with the example in \Cref{siex}.\end{remark}


The following lemma  {relates the semistable dual graph $\gss$,  the plumbing graph associated with~$h^e$, and the Nielsen-Thurston graph associated to~$h$. }


	\begin{lemma}\label{lem:homtopy_equivalence_graphs}
		Let $f: (X,0) \to (\C,0)$ define a germ of curve singularity on $(X,0)$, let $h: F \to F$ be the geometric monodromy associated with $f$ and let $e\in \NN$ be the smallest natural number such that $h^e$ is a composition of right-handed Dehn twists around disjoint simple closed curves. Then the following graphs have the same homotopy type (when seen as $1$-dimensional CW-complexes):
		\begin{enumerate}[label=\rm(\roman{enumi})]
			\item \label{it:i} The dual graph $\gss$ of the semistable reduction.
			\item \label{it:ii} The plumbing graph associated to the mapping torus of $h^e$.
			\item \label{it:iii} The Nielsen-Thurston graph associated to $h$.
		\end{enumerate}
	\end{lemma}

\begin{proof}
	After the base change $\sigma: D_{\delta^{1/e}} \to D_\delta$,  given by $t \mapsto t^e$ and used to the construction of the semistable reduction, the monodromy $h$ becomes the homeomorphism $h^e$ which is also pseudo-periodic. One can now find the central fiber $(f^{(e)})^{-1}(0)$  of the semistable reduction by looking at the mapping torus of its monodromy and computing the corresponding plumbing graph. To get to the dual graph of the semistable reduction, one may blow-down the rational curves of self-intersection $-1$ that intersect the other curves in at most $2$ points, and this process clearly does not change the homotopy type of the graph. Actually, this is the equivalent to reduce the plumbing graph to its minimal form. This covers the equivalence of \ref{it:i}
	and \ref{it:ii}.
	
	To establish the equivalence with the previous two graphs with \ref{it:iii} we just observe that $h^e$ is the identity outside a tubular neighborhood of the curves that gives the Nielsen decomposition of~$h$. This tells us that its Nielsen-Thurston graph has as many vertices as nodes has the plumbing graph associated to the mapping torus of $h^e$. Also, each curve of the Nielsen decomposition of $h^e$ is an orbit in itself so there is a bijection between the tori in the graph manifold structure of the plumbing graph associated to the mapping torus of $h^e$ and the curves in the Nielsen decomposition of $h^e$. So this establishes the equivalence between  \ref{it:ii} and \ref{it:iii}. 
\end{proof}

\begin{remark}
In fact, as we have taken as $\gss$ the dual graph of the semistable reduction coming from a minimal $\mathbb{Q}$-resolution, this graph is actually equal to the Nielsen-Thurston graph associated to~$h$.
\end{remark}

\subsection{Algebraic invariants of the monodromy}\label{subsec:alg_inv_mon}
\mbox{}

Now we explain how the  characteristic polynomial $\Delta(t)$ of $h_\ast$ on $H_1(F)$ and the 
	characteristic polynomial $\Delta_2(t)$ of   $h_\ast$ on $H_1(F)/ \ker 
	(h^e_\ast - {\rm Id})$ can be expressed in terms of the decorated dual resolution graph
	$\Gamma_f$, when $M$ is a rational homology sphere.

	
Since $M$ is assumed to be a rational homology sphere, the $g_v$ is  zero for each vertex $v$.
We recall that a vertex $v\in\calV$ is a {{\em node}}  if 
$\chi_v<0$ 
and we say that it is a {{\em separating node}}  if at least 
two of the components of $\Gamma \setminus \{v\}$ contain arrows. Let 
$\mathcal{V}'$ be the set of  {separating nodes}. For each $e \in 
	\mathcal{E}$ denote by $m_e$ be the greatest common divisor of the 
	multiplicities of the two vertices at the ends of $e$. Moreover, let 
	$\mathcal{E}'$ be the set formed by a choice of one edge for each path or chain 
	joining two  {separating nodes}. Finally, let $d$ the greatest common divisor of the multiplicities of the arrowheads of $\G_f$. 
	The aforementioned characteristic 
	polynomials are given by the following expressions
	\begin{gather*}
	\Delta(t)= (t^d-1) \prod_{v \in \mathcal{V}} (t^{m_v}-1)^{-\chi_v}\ \text{\cite[Theorem 11.3]{en}},\\
	\Delta_2(t)= (t^d-1) \frac{\prod_{e \in \mathcal{E}'} (t^{m_e}-1)}{\prod_{v 
			\in \mathcal{V}'} (t^{m_v}-1)}\ \text{\cite[Theorem 14.1]{en}}.
	\end{gather*}

	\begin{remark}	In particular, it is not 
	difficult to establish that the number of Jordan blocks of $h_\ast$ is given by 
	the first Betti number $b_1(\gss )$, see \cref{sec:ex} and  \Cref{re:inducv}. 	A further reference
is the main result of \cite{G} for generalizations to higher dimensions. \end{remark}

	\section{Explicit computations and examples}\label{sec:ex}
	
Let $F$ be the fiber of a locally trivial fibration of a graph manifold $M$ fibered over $\MS^1$ such that $F$ is transverse to all the circle fibers of the circle bundles that form $M$ and such that the monodromy of the locally trivial fibration is pseudo-periodic.
	Let us denote $\calC^+=\calC \cup \partial F$, and let $F_\calC$ be the closed complement of a 
	tubular neighborhood $N_\calC$ of $\calC^+$ in $F$. One can consider the long exact 
	sequence of the pair $(F,F_\calC)$
	\begin{equation}\label{les}
	0 \rightarrow H_1(\calC) \rightarrow H_1(F_\calC) \rightarrow H_1(F) \rightarrow H_0(\calC) 
	\rightarrow H_0(F_\calC) \rightarrow H_0(F) \rightarrow 0,
	\end{equation}
applying the Excision Formula to $\mathring{F}_\calC$ and taking into account that for an annulus $T$
with boundary components $C_1,C_2$, we have $H_\bullet(T,C_i)=0$ and $H_\bullet(T,\partial T)\cong \tilde{H}_{\bullet-1}(C_i)$.
Additionally one can define the increasing weight filtration
\[
W_0:= \im  (H_1(\calC^+) \rightarrow H_1(F)) \subset W_1:= \im  (H_1(F_\calC)\rightarrow H_1(F)) \subset
 W_2:=H_1(F).
\]

Recall that $W_1 \subset \ker N'$ because $F_C$ is a union of periodic parts of the monodromy. Using the twist formula \Cref{tf}, one can deduce that $W_1 = \ker N'$, see \cite[Theorem 10.3.8\ii]{N} {for the case of plane curves and notice that the argument also works in general}.

As a direct consequence from the long exact seqence of the pair $(F,\partial F)$, we obtain that $H_1(F; \ZZ)/\ker N'$ can be naturally identified as a subgroup of 
$H_1(F,\partial F; \ZZ)/\ker N$. If we denote as $\tilde{Q}'$ the bilinear map induced by $Q'$ on $H_1(F; \ZZ)/\ker N'$ it follows from the very definition that 
$\tilde{Q}'$ is the restriction of $\tilde{Q}$ to $H_1(F; \ZZ)/\ker N'$.
	
Using the equivalence between the dual graph $\gss $ of the semistable reduction and the Nielsen-Thurston graph established in \Cref{lem:homtopy_equivalence_graphs}, we are going to reduce the computation of the quadratic form $\tilde{Q}$ on $H_1(F, \partial F; \ZZ)/\ker N$ (resp. its restriction to $H_1(F; \ZZ)/\ker N'$) to a computation on $H_1(\gss , \calD; \ZZ)$ the first relative homology group of 1-chains of the dual graph of the semistable reduction relative to the set of its arrows $\calD$ (resp.  $H_1(\gss ; \ZZ)$)\label{calD}.

Firstly, we identify the group 
$H_1(F, F_C)$ with the group $C_1(\gss )$,  and with the group $H_1(\gss , \Vss)$, where $\Vss$ denotes the set of nodes of $\gss$. We also identify $H_0(F_C; \ZZ)$ with the group $C_0(\gss )$. Secondly, the map $H_1(F, F_C; \ZZ) \rightarrow H_0(F_C; \ZZ)$ from \eqref{les} is identified with the boundary map $\delta_1 : C_1(\gss ) \rightarrow C_0(\gss )$.
	
From \eqref{les}, we get the short exact sequence
\[
0 \rightarrow W_1=\ker N' \rightarrow W_2=H_1(F; \ZZ) \rightarrow \ker(C_1(\gss ) \rightarrow C_0(\gss )) \rightarrow 0,
\]
and the isomorphism 
\[
\frac{H_1(F; \ZZ)}{\ker N'}=\frac{W_2}{W_1} \rightarrow H_1(\gss ).
\]
	Analogously, 
	$H_1(F, \partial F; \ZZ)/\ker N$ is isomorphic to $H_1(\gss , \calD; \ZZ)$. Therefore, we can consider $\tilde{Q}$ as a form defined on  $H_1(\gss , \calD; \ZZ) \times H_1(\gss , \calD; \ZZ)$.

	To perform particular computations, we use that the image in $H_1(\gss , \calD; \ZZ)$ of the intervals $I_{i,b}$,  introduced in the proof of \Cref{thm:def_pos}, are supported on bamboos of $\gss $ connecting two nodes or a node and an arrow.

	\medskip
	
	Let us compute some examples. The reader can check these computations, and find supplementary details in the  link \href{https://github.com/enriqueartal/QuadraticFormSingularity}{https://github.com/enriqueartal/QuadraticFormSingularity} using 
	\texttt{Sagemath}, eventually within \texttt{Binder}.
	
	\medskip
	
	\Cref{ex:ACampo1} and \Cref{ex:ACampo2} are intended to illustrate the method of computation.
	
	 
	\begin{example}\label{ex:ACampo1}
		The A'Campo double $(7,6)$-cusp $f=(x^6+y^7)(x^7+y^6)$ defines a singular point of a plane curve with Milnor number $\mu=131$, and its characteristic polynomials are
\[
		\Delta(t)=\frac{(t-1)(t^{78}-1)^2}{(t^{13}-1)^2}, \text{ and } \Delta_2(t)=\frac{t^6-1}{t-1},
\]
see \Cref{fig:2branches67} for the dual resolution graph and the Nielsen-Thurston graph.
		
		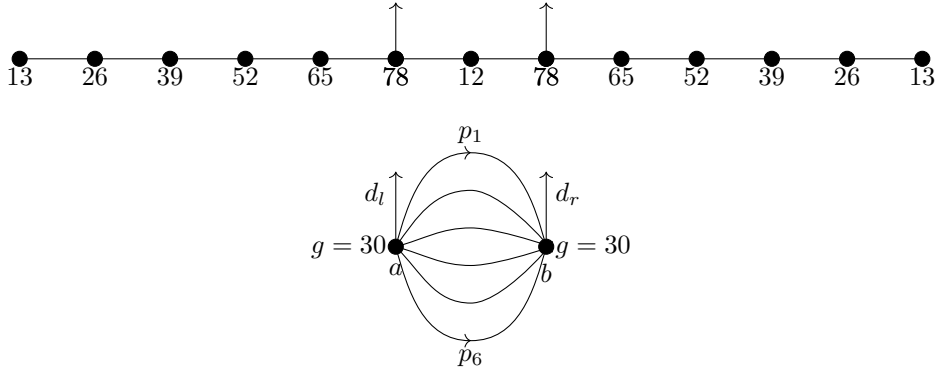
\begin{figure}[ht]
			\begin{center}
							\begin{tikzpicture}[vertice/.style={draw,circle,fill,minimum size=0.2cm,inner sep=0}]
				\coordinate (U) at (0,0);
				\coordinate (A1) at (1,0);
				\coordinate (A2) at (2,0);
				\coordinate (A3) at (3,0);
				\coordinate (A4) at (4,0);
				\coordinate (A5) at (5,0);
				\coordinate (A6) at (6,0);
				\coordinate (B1) at (-1,0);
				\coordinate (B2) at (-2,0);
				\coordinate (B3) at (-3,0);
				\coordinate (B4) at (-4,0);
				\coordinate (B5) at (-5,0);
				\coordinate (B6) at (-6,0);
				\coordinate (F1) at (1,0.75);
				\coordinate (F2) at (-1,0.75);

				\node[vertice] at (U) {};
				\node[below] at (U) {$12$};
				\node[vertice] at (A1) {};
				\node[below] at (A1) {$78$};
				\node[vertice] at (B1) {};
				\node[below] at (B1) {$78$};
				\node[vertice] at (A2) {};
				\node[below] at (A2) {$65$};
				\node[vertice] at (A3) {};
				\node[below] at (A3) {$52$};
				\node[vertice] at (A4) {};
				\node[below] at (A4) {$39$};
				\node[vertice] at (A5) {};
				\node[below] at (A5) {$26$};
				\node[vertice] at (A6) {};
				\node[below] at (A6) {$13$};
				\node[vertice] at (A1) {};
				\node[below] at (A1) {$78$};
				\node[vertice] at (B1) {};
				\node[below] at (B1) {$78$};
				\node[vertice] at (B2) {};
				\node[below] at (B2) {$65$};
				\node[vertice] at (B3) {};
				\node[below] at (B3) {$52$};
				\node[vertice] at (B4) {};
				\node[below] at (B4) {$39$};
				\node[vertice] at (B5) {};
				\node[below] at (B5) {$26$};
				\node[vertice] at (B6) {};
				\node[below] at (B6) {$13$};
				
				\draw (U)--(A1)--(A2)--(A3)--(A4)--(A5)--(A6);
				\draw (U)--(B1)--(B2)--(B4)--(B4)--(B5)--(B6);
				\draw[->] (A1)--(F1);
				\draw[->] (B1)--(F2);
				
				\coordinate (C1) at (0,-3.75);
				\coordinate (C2) at (0,-3.25);
				\coordinate (C3) at (0,-2.75);
				\coordinate (C4) at (0,-2.25);
				\coordinate (C5) at (0,-1.75);
				\coordinate (C6) at (0,-1.25);
				\coordinate (D1) at (1,-2.5);
				\coordinate (D2) at (-1,-2.5);
				\coordinate (G1) at (1,-1.5);
				\coordinate (G2) at (-1,-1.5);
				\draw[->] (D2)to[out=-75,in=180] node[pos=1,below] {$p_{6}$} (C1) ;
				\draw (C1) to[out=0,in=-105](D1); 
				\draw (D2)to[out=-45,in=180](C2)to[out=0,in=45](D1);
				\draw (D2)to[out=-15,in=180](C3)to[out=0,in=15](D1);
				\draw (D2)to[out=15,in=180](C4)to[out=0,in=-15](D1);
				\draw (D2)to[out=55,in=180](C5)to[out=0,in=-45](D1);
				\draw[->] (D2)to[out=75,in=180] node[pos=1,above] {$p_{1}$} (C6) ;
				\draw (C6) to[out=0,in=105](D1); 
				\node[vertice] at (D1) {};
				\node[right] at (D1) {$g=30$};
				\node[below=3pt] at (D1) {$b$};
				\node[vertice] at (D2) {};
				\node[left] at (D2) {$g=30$};
				\node[below=3pt] at (D2) {$a$};
				\draw[->] (D1)--(G1) node[right, pos=.7] {$d_r$};
				\draw[->] (D2)--(G2) node[left, pos=.7] {$d_l$};
				\end{tikzpicture}
				\caption{Dual graph of the embedded resolution of $f$, and dual graph of its semistable reduction from a minimal $\mathbb{Q}$-resolution.}
				\label{fig:2branches67}
			\end{center}
		\end{figure}
		
		Let us fix  basis of $H_1(\gss )$ and $H_1(\gss , \calD)$. Let us label the leftmost (resp. rightmost) vertex of  the Nielsen-Thurston graph by $a$ (resp. $b$).
		Label $p_1,\dots,p_6$ (top to bottom) the oriented edges (from $a$ to $b$) 
The set 
		\[
		\{\sigma_i= p_{i}-p_{i+1} \, \mid \, i=1, \dots , 5\} 
		\]
		is a basis of $H_1(\gss )$. A basis of $H_1(\gss , \calD)$ is obtained by attaching to the former set the element  $\sigma_6= -d_l +p_1+d_r$.

		Let us show that the quadratic form is represented by the following  matrix (see also \url{https://github.com/enriqueartal/QuadraticFormSingularity}).
\[
\left ( \begin{array}{rrrrr|r}
2 & -1 & 0 & 0 & 0 & 1 \\
-1 & 2 & -1 & 0 & 0 & 0 \\
0 & -1 & 2 & -1 & 0 & 0 \\
0 & 0 & -1 & 2 & -1 & 0 \\
0 & 0 & 0 & -1 & 2 & 0 \\\hline
1 & 0 & 0 & 0 & 0 & 3		\end{array} \right ).
\]
		Remark that it is enough to compute the $(i, j)$-entries  with $j = i-1,i, i+1$, and the $(1, 6)$-entry.

		In this case we take $e=78$, and there are three orbits of curves or annuli:
		\begin{itemize}
			\item The orbits $\calC_l$ (resp. $\calC_r$) corresponding to the 
			left arrow $d_l$ (resp. $d_r$) with $d=1$. Denote by $C_\ast$ (resp. $T_\ast$) for $\ast = l,r$ 
			the unique curve (resp. annulus) of the orbit. Assume that the orientation of $C_\ast$ is determined by 
			$d_\ast$  and \Cref{lem:twist_number}\ref{lem:twist_number_i}. 
			By \Cref{tf}, $\scn(h,\calC_\ast)=1/78$, and using \Cref{lem:twist_number}\ref{lem:twist_number_i} 
			$\scn(h^e, T_\ast)=1$.
			
			\item The orbit $\calC$ corresponding  to the bamboo of three vertices separating the two arrows of the dual resolution graph $\Gamma$ with  $d=6$. 
			Denote by $C_i$ (resp. $T_i$), with $i=1, \dots, 6$, the components of $\calC$ (resp. 
			of the orbit $\calT$ of annuli associated to $\calC$). Assume that the orientations of the curve $C_i$ 
			are determined by $\sigma_i$ and \Cref{lem:twist_number}\ref{lem:twist_number_i}.  
			As in the previous case  $\scn(h, \calC)= 1/13$, and $\scn(h^e,T_i)=1$. 
		\end{itemize} 
		
		For any of the above annulus $T$, let us denote the boundary component of $T$ closer to the boundary component of $F$ corresponding to $d_l$ by $\partial_1 T$, and the other boundary component by $\partial_2T$.

		The basis element $\sigma_i$ lifts to a simple closed curve $\gamma_i$ that intersect the annulus $T_i$ (resp. $T_{i+1}$) in a single oriented interval $I_{i} \subset I_i^+$ (resp. $I_{i+1} \subset I_{i+1}^-$). Clearly,
		\[
		Q(\gamma_i, \gamma_i)=\langle 
		C_i+C_{i+1}
		, I_{i}+I_{i+1}\rangle =2.
		\]
		The curve $\gamma_{i-1}$ intersects the annulus $T_i$ in an interval $J_i  \subset I_i^-$ and does not intersect $T_{i+1}$. 
		The basis element $\gamma_{i+1}$ intersects the annulus $T_{i+1}$ in an interval $J_{i+1}  \subset I_i^+$ and does not intersect~$T_{i}$.   Hence,
		\[
		Q(\gamma_i, \gamma_k) = \langle 
		C_i+C_{i+1}
		, J_{k}\rangle  = -
		1,\qquad k=i-1, i+1.
		\]
		Finally, the basis element $\gamma_{6}$ intersects the annuli $T_l$, $T_1$, and $T_r$ in intervals $K_l \subset I^+_l, K_1 \subset I^+_1$, and $K_r \subset I^+_r$. Hence,
		\[
		Q(\gamma_6, \gamma_1) =\langle 
		C_l+C_1+C_r
		, K_1\rangle =
		1,\qquad Q(\gamma_6, \gamma_6) =\langle 2(%
		C_l+C_1+C_r
		, K_l+K_1+K_r\rangle =
		3.
		\] 
Observe that, as proven in \Cref{cor:even} the part corresponding to the absolute homology is 
even, but the part corresponding to the relative cycle is not.
	\end{example}

	\begin{example}\label{ex:ACampo2}
		The germ $g=((y^2+x^3)^2+x^5y)((x^2+y^3)^2+xy^5)$ defines a singular point of plane curve with Milnor number $\mu=63$, and characteristic polynomials
\[
\Delta(t)=\frac{(t-1)(t^{20}-1)^2(t^{42}-1)^2}{(t^{10}-1)^2(t^{21}-1)^2}, \hbox{ and } \Delta_2(t)=\frac{t^4-1}{t-1},
\]
see \Cref{fig:2branches2pares} for the dual resolution graph and its Nielsen-Thurston graph.

		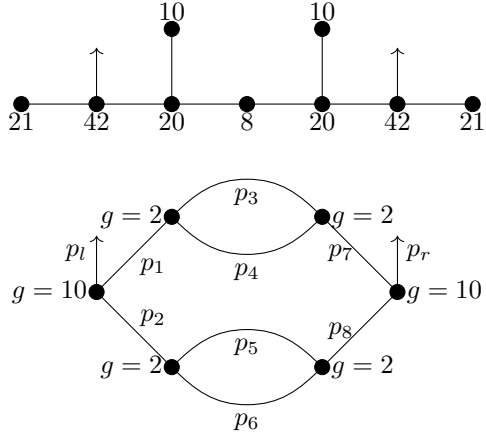
\begin{figure}[ht]
			\begin{center}
\begin{tikzpicture}[vertice/.style={draw,circle,fill,minimum size=0.2cm,inner sep=0}]
\coordinate (U) at (0,0);
\coordinate (A1) at (1,0);
\coordinate (A2) at (2,0);
\coordinate (A3) at (3,0);
\coordinate (B1) at (-1,0);
\coordinate (B2) at (-2,0);
\coordinate (B3) at (-3,0);
\coordinate (A4) at (1,1);
\coordinate (B4) at (-1,1);
\coordinate (F1) at (2,0.75);
\coordinate (F2) at (-2,0.75);

\node[vertice] at (U) {};
\node[below] at (U) {$8$};
\node[vertice] at (A1) {};
\node[below] at (A1) {$20$};
\node[vertice] at (B1) {};
\node[below] at (B1) {$20$};
\node[vertice] at (A2) {};
\node[below] at (A2) {$42$};
\node[vertice] at (B2) {};
\node[below] at (B2) {$42$};
\node[vertice] at (A3) {};
\node[below] at (A3) {$21$};
\node[vertice] at (B3) {};
\node[below] at (B3) {$21$};
\node[vertice] at (A4) {};
\node[above] at (A4) {$10$};
\node[vertice] at (B4) {};
\node[above] at (B4) {$10$};

\draw (U)--(A1)--(A2)--(A3);
\draw (U)--(B1)--(B2)--(B3);
\draw (A1)--(A4);
\draw (B1)--(B4);
\draw[->] (A2)--(F1);
\draw[->] (B2)--(F2);

\coordinate (C1) at (0,-1);
\coordinate (C2) at (0,-2);
\coordinate (C3) at (0,-3);
\coordinate (C4) at (0,-4);
\coordinate (D1) at (1,-1.5);
\coordinate (D2) at (-1,-1.5);
\coordinate (D3) at (1,-3.5);
\coordinate (D4) at (-1,-3.5);
\coordinate (E1) at (2,-2.5);
\coordinate (E2) at (-2,-2.5);
\coordinate (G1) at (2,-1.75);
\coordinate (G2) at (-2,-1.75);

\draw (E2) -- node[below=5pt,pos=.75] {$p_1$} (D2)to[out=45,in=180]
node[below=0pt,pos=1] {$p_3$}(C1) to[out=0,in=135] (D1)--node[below=0pt,pos=.25] {$p_7$}(E1) ;
\draw (D2)to[out=-45,in=180] node[below=0pt,pos=1] {$p_4$}(C2) to[out=0,in=-135] (D1);
\draw (E2)--node[above=5pt,pos=.75] {$p_2$}(D4)to[out=-45,in=180] node[below=0pt,pos=1] {$p_6$}(C4) to[out=0,in=-135] (D3)--node[above=0pt,pos=.25] {$p_8$}(E1);
\draw (D4)to[out=45,in=180] node[below=0pt,pos=1] {$p_5$}(C3) to[out=0,in=135] (D3);

\node[vertice] at (D1) {};
\node[vertice] at (D2) {};
\node[vertice] at (D3) {};
\node[vertice] at (D4) {};
\node[vertice] at (E1) {};
\node[vertice] at (E2) {};

\node[right] at (D1) {$g=2$};
\node[left] at (D2) {$g=2$};
\node[right] at (D3) {$g=2$};
\node[left] at (D4) {$g=2$};
\node[right] at (E1) {$g=10$};
\node[left] at (E2) {$g=10$};

\draw[->] (E1)--(G1) node[below right] {$p_r$};
\draw[->] (E2)--(G2) node[below left] {$p_l$};
\end{tikzpicture}
\caption{Dual graph of the embedded resolution of $g$, and the dual graph of its semistable reduction from a minimal $\mathbb{Q}$-resolution.}
				\label{fig:2branches2pares}
			\end{center}
		\end{figure}
		
		Let us fix a basis of $H_1(\gss )$ and $H_1(\gss , \calD)$. 
		Let us label by $p_i$ the edges (not arrows), from top to bottom and for left to right.
		Similary, we label the other 4 edges by $p_i$, with $i= 9, \dots, 12$. Denote the left (resp. right) arrow by $p_l$ (resp. $p_r$) with its natural orientation. The set 
\[
\left \{ 
		\begin{array}{c}\sigma_1= 
		p_3-p_4,\\ 
		\sigma_2= 
		p_1+p_4+p_7-p_8-p_5-p_2,\\
		\sigma_3= 
		p_5-p_6
		\end{array} 
		\right \}
\]
is a basis of $H_1(\gss )$. A basis of  $H_1(\gss , \calD)$ is obtained adding the element 
		$\sigma_4=
		-p_l+p_1+p_3+p_7+p_r$. 
		
		Let us show that the  quadratic form is represented by the matrix 
\[
\left ( \begin{array}{rrr|r}
42 & -21 & 0 & 21 \\
-21 & 46 & -21 & 2 \\
0 & -21 & 42 & 0 \\\hline
21 & 2 & 0 & 43		\end{array} \right ).
\]	
		In this case, we take $e=
		420$ and there are five orbits of curves or annuli:
		\begin{itemize}
			\item The orbits $\calC_l$ (resp. $\calC_r$) corresponding to the left-hand arrow $d_l$ (resp. right-hand arrow $d_r$) with $d=1$. Denote by $C_\ast$ (resp. $T_\ast$) for $\ast = l,r$ the unique curve (resp. annulus) of the orbit. Assume that the orientation of $C_\ast$ is determined by $p_\ast$ and  \Cref{lem:twist_number}\ref{lem:twist_number_i}. We have that  $\scn(h,\calC_\ast)=\frac{1}{42}$, and $\scn(h^e, T_\ast)=
			10$.
			\item The orbits $\calC_i$ with $i=1,2$ corresponding  to the bamboo of the two vertices with multiplicities 42 and  20 of  the dual resolution graph $\Gamma$ with $d=2$.  Denote by $C_{i,j}$ (resp. $T_{i,j}$), with $i=1, 2$, and $j=1, 2$, the components of $\calC$ (resp. of the orbit $\calT$ of annuli associated to $\calC$). Assume that the orientations of the curve $C_{i,j}$ are determined by $\sigma_2$ and \Cref{lem:twist_number}\ref{lem:twist_number_i}. We have that 
			$\scn(h, \calC_i)= \frac{1}{210}$, and $\scn(h^e,T_{i,j})=
			1$.
			\item The orbit $\calC_3$ corresponding  to the central bamboo of the three vertices with multiplicities 20, 8, and  20 of  the dual resolution graph $\Gamma$ with $d=4$.    Denote by $C_{3,k}$ (resp. $T_{3,k}$), with $k=1, \dots,  4$, the components of $\calC_3$ (resp. of the orbit $\calT_3$). Assume that the orientations of the curve $C_{3,k}$ are determined by $\sigma_1$ and $\sigma_3$ and \Cref{lem:twist_number}\ref{lem:twist_number_i}. We have that $\scn(h, \calC)=\frac{1}{10}$, and $\scn(h^e,T_{3,k})=
			21$.
		\end{itemize}
		
		For any of the above annulus $T$, we label  the boundary component of $T$ as in the previous example.
		
		The base element $\sigma_1$ lifts to a curve $\gamma_1$ in $F$ intersecting the annuli $T_{3,1}$ and $T_{3,2}$ in intervals $I_{3,1} \subset I_{3,1}^+$, and $I_{3,2} \subset I_{3,2}^-$. Then, we have that $Q(\gamma_1, \gamma_1)=\langle 21(
		C_{3,1}+C_{3,2})
		, I_{3,1}+I_{3,2}\rangle =
		42$.
		
		The base element  $\sigma_2$ lifts to a curve $\gamma_2$ in $F$ intersecting the annulus $T_{1,i}$ in an interval $J_{1,1} \subset I_{1,1}^+$, the annulus $T_{3,2}$ in an interval $J_{3,2} \subset I_{3,2}^+$, and the annulus $T_{2,1}$ in an interval $J_{2,1} \subset I_{2,1}^+$. Then,
\[
		Q(\gamma_2, \gamma_2)=\langle 
		C_{1,1} +C_{2,1}+ C_{2,2}+C_{1,2}
		+
		21(C_{3,2}+ C_{3,3}), I_{1,1}+I_{2,2}+I_{3,1}+I_{3,2}+I_{2,3}+I_{1,2}\rangle =%
		46,
\]		
and $Q(\gamma_1, \gamma_2) =\langle
21(C_{3,1}+C_{3,2}),J_{3,2}\rangle =-
21$.
		
		The base element $\sigma_4$ lifts to a curve $\gamma_4$ that intersects  the annulus $T_l$ in an interval $K_l \subset I_l^+$, the annulus $T_{1,1}$ in an interval $K_{1,1} \subset I_{1,1}^+$, the annulus $T_{3,1}$ in an interval $K_{3,1} \subset I_{3,1}^+$, and the annulus $T_l$ in an interval $K_l \subset I_l^+$. Then, 
\begin{gather*}
Q(\gamma_4, \gamma_4)= \langle 
10C_l + 
C_{1,1} + 
21C_{3,1}+
C_{2,1} + 
10C_r, I_l +I_{1,1}+I_{3,1}+I_{2,1}+ I_r\rangle = 
43,\\
Q(\gamma_1, \gamma_4)= \langle 
21(C_{3,1}+C_{3,2}), I_l +I_{1,1}+I_{3,1}+I_{2,1}+ I_r\rangle =
{21},\\
Q(\gamma_2, \gamma_4)= \langle 
C_{1,1} +C_{2,1}+ C_{2,2}+C_{1,2}
-
21(C_{3,2}+ C_{3,3}), I_l +I_{1,1}+I_{3,1}+I_{2,1}+ I_r\rangle =
2.
\end{gather*}

	\end{example}

	\Cref{ex:decomp} shows that the restriction of $\tilde{Q}$ to  $H_1(F, \mathbb{Z})/ \ker N'$ can be decomposible.
		
	\begin{example}\label{ex:decomp}	
The 	polynomial	$h=(x+y)(x-y)(x^2+y^3)(y^2+x^3)$ defines a singular point with Milnor number $\mu=27$,
\[
\Delta(t)=\frac{(t-1)(t^{14}-1)^2(t^6-1)^2}{(t^7-1)^2}, \hbox{ and } \Delta_2(t)=(t-1)^2,
\]
see \Cref{fig:4branches} for the dual resolution graph and its Nielsen-Thurston graph.

		\begin{figure}[ht]
			\begin{center}
				\begin{tikzpicture}[vertice/.style={draw,circle,fill,minimum size=0.2cm,inner sep=0}]
				\coordinate (U) at (0,0);
				\coordinate (A1) at (1,0);
				\coordinate (A2) at (2,0);
				\coordinate (B1) at (-1,0);
				\coordinate (B2) at (-2,0);
				\coordinate (F1) at (1,0.75);
				\coordinate (F2) at (-1,0.75);
				\coordinate (F3) at (0.5,0.75);
				\coordinate (F4) at (-0.5,0.75);

				\node[vertice] at (U) {};
				\node[below] at (U) {$6$};
				\node[vertice] at (A1) {};
				\node[below] at (A1) {$14$};
				\node[vertice] at (B1) {};
				\node[below] at (B1) {$14$};
				\node[vertice] at (A2) {};
				\node[below] at (A2) {$7$};
				\node[vertice] at (B2) {};
				\node[below] at (B2) {$7$};
				
				\draw (U)--(A1)--(A2);
				\draw (U)--(B1)--(B2);
				\draw[->] (A1)--(F1);
				\draw[->] (B1)--(F2);
				\draw[->] (U)--(F3);
				\draw[->] (U)--(F4);

\begin{scope}[yshift=-5mm]
				\coordinate (Q) at (0,-1.5);
				\coordinate (C1) at (-1,-1.5);
				\coordinate (C2) at (1,-1.5);
				\coordinate (C3) at (0,-2);
				\coordinate (G1) at (-1,-0.75);
				\coordinate (G2) at (1,-0.75);
				\coordinate (G3) at (-0.5,-0.75);
				\coordinate (G4) at (0.5,-0.75);

				\node[vertice] at (Q) {};
				\node[below] at (C3) {$g=4$};
				\node[vertice] at (C1) {};
				\node[left] at (C1) {$g=3$};
				\node[vertice] at (C2) {};
				\node[right] at (C2) {$g=3$};
				
				\draw (0,-1.5) arc (0:-180:0.5);
				\draw (1,-1.5) arc (0:-180:0.5);
				
				\draw (C1)--(C2);
				
				\draw[->] (C1)--(G1) node[left] {$p_l$};
				\draw[->] (C2)--(G2) node[right] {$p_r$};
				\draw[->] (Q)--(G3) node[above] {$p_{c,l}$};
				\draw[->] (Q)--(G4) node[above] {$p_{c,r}$};
				\end{scope}
				\end{tikzpicture}
				\caption{Dual graph of the embedded resolution of $h$, and 
				the dual graph of its semistable reduction from a minimal $\mathbb{Q}$-resolution.}
				\label{fig:4branches}
			\end{center}
		\end{figure}
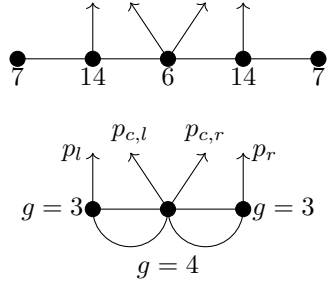

Let us label the vertex of the Nielsen-Thurston graph, from left to right and top to botton, by $a, b, c$. Let us denote the oriented edges from $a$ to $b$ by $p_i$ with $i=1,2$ and the oriented edges from $b$ to $c$ by $p_i$, with $i=3,4$. Let us label the arrows, from left to right, by $p_l, p_{cl}, p_{cr}, p_r$. The set $\{\sigma_1 = p_1-p_2, \sigma_2=p_3-p_4\}$ is a basis of $H_1(\gss )$. A basis of  $H_1(\gss , \calD)$ is obtained adding the elements 
\[
\sigma_{3}=-p_l+ p_1+p_{cl}, \sigma_{4}=-p_{c,l}+p_{c,r}, \sigma_{5}=-p_{c,r}+p_3+p_r.
\] 
Let us show that the  quadratic form is represented by the matrix 
\[
\left ( \begin{array}{rr|rrr}
		2&0&1&0&0\\
		0&2&0&0&1\\
		\hline
		1&0&11&-7&0\\
		0&0&-7&14&-7\\
		0&1&0&-7&11
		\end{array} \right ).
\]
In this case $e=42$ and there are six orbits of curves or annuli:
\begin{itemize}
			\item The orbits $\calC_\ast$ for $\ast=l,cl,cr,r$ corresponding to the arrow $d_\ast$  with $d=1$. Denote by $C_\ast$ (resp. $T_\ast$)  the unique curve (resp. annulus) of the orbit. Assume that the orientation of $C_\ast$ is determined by $p_\ast$ and 
			\Cref{lem:twist_number}\ref{lem:twist_number_i}. We have that  $s(h,\calC_\ast)=1/14$, $s(h^e, T_\ast)=3$, for $\ast = l,r$, and that $\scn(h,\calC_\ast)=1/6$, and $\scn(h^e, T_\ast)=7$ for $\ast=cl,cr$.
			\item The orbit $\calC_i$ with $i=1,2$ corresponding  to the bamboo of the two vertices with multiplicities 14 and  6 of  the dual resolution graph $\Gamma$ with $d=2$.  Denote by $C_{i,j}$ (resp. $T_{i,j}$), with $i=1, 2$, and $j=1,  2$, the components of $\calC$ (resp. of the orbit $\calT$ of annuli associated to $\calC$). Assume that the orientations of the curve $C_{i,j}$ are determined by $\gamma_1$ and $\gamma_3$ and \Cref{lem:twist_number}\ref{lem:twist_number_i}. We have that 
			$\scn(h, \calC_i)= \frac{1}{2}1$, and $\scn(h^e,T_{i,j})=1$.
		\end{itemize}
		
		The base element $\sigma_i$, for $i=1,2$ lifts to a curve $\gamma_i$ in $F$ intersecting the annuli $T_{i,1}$ and $T_{i,2}$ in intervals $I_{i,1} \subset I_{i,1}^+$, and $I_{i,2} \subset I_{i,2}^-$. Then, we have that $Q(\gamma_1, \gamma_1)=\langle C_{i,1}+C_{i,2}, I_{i,1}+I_{i,2}\rangle =2$.
		
		The base element $\sigma_3$ lifts to a curve $\gamma_3$ in $F$ intersecting the annulus $T_l$ in an interval $J_l\subset I_l^-$, the annulus $T_{1,1}$ in an interval $J_{1,1} \subset I_{1,1}^+$, and the annulus $T_{cl}$ in an interval $J_{cl} \subset I_l^+$. We have that
		$Q(\gamma_3, \gamma_3)=\langle 3C_l + C_{1,1} + 7 C_{cl}, J_l + J_{1,1} + J_{cl}\rangle =11,$ and
		$Q(\gamma_1, \gamma_3)=\langle C_{1,1}+C_{1,2}, J_l + J_{1,1} + J_{cl}\rangle =1$.
		
		The base element $\sigma_4$ lifts to a curve $\gamma_4$ in $F$ intersecting the annulus $T_{cl}$ in an interval $K_{c,l} \subset I_{cl}^+$, and the annulus $T_{cr}$ in an interval $K_{cr} \subset I_{cr}^-$. We have that
		$Q(\gamma_4, \gamma_4) = \langle 7(C_{cl}+C_{cr}), K_{cl}+K_{cr}\rangle =14,$ and $Q(\gamma_3, \gamma_4)=\langle 3C_l + C_{1,1} + 7 C_{cl}, K_{cl}+K_{cr}\rangle = -7.$
		
	\end{example}

	\Cref{ex:SSS} and \Cref{ex:DBM} are of historical interest in Singularity Theory.
		
\begin{example}\label{ex:SSS} Steenbrink, Stevens, and Schrauwen showed that spectral pairs are not a complete invariant of the topological type of plane curve singularities, see \cite[Example 5.4]{Steen} and~\cite{Kaen}. Let us  consider the particular curves  given by polynomials $f_{11;00}$ and $f_{10;10}$, where 
\[
f_{kl;mn}=((y-x^2)^2-x^{5+k})((y+x^2)^2-x^{5+l})((x-y^2)^2-y^{5+m})((x+y^2)^2-y^{5+n}).
\]
Our computations below show that the corresponding quadratic forms $\tilde{Q}$ have different determinants so they can't be similar
(even over $\mathbb{Q}$ since the quotient of the determinants is not a perfect square).
	
\Cref{fig:sss_example} shows the dual resolution graphs and Nielsen-Thurston graphs of the curves defined by  $f_{11;00}$ (left) and $f_{10;10}$ (right); the weights in the 
Nielsen-Thurston graph are the multiplicities and the genera.  
	
	\begin{figure}[ht]
		\begin{center}
			\begin{tikzpicture}[vertice/.style={draw,circle,fill,minimum size=0.2cm,inner sep=0}]
			\coordinate (U) at (-3,0);
			\coordinate (A1) at (-2,0);
			\coordinate (A2) at (-1,1);
			\coordinate (A3) at (-1,-1);
			\coordinate (B1) at (-4,0);
			\coordinate (B2) at (-5,1);
			\coordinate (B3) at (-5,-1);
			\coordinate (C2) at (0,1);
			\coordinate (C3) at (0,-1);
			
			\coordinate (F1) at (-6,0.75);
			\coordinate (F2) at (-6,1.25);
			\coordinate (F3) at (-6,-0.75);
			\coordinate (F4) at (-6,-1.25);
			\coordinate (F5) at (-1,1.75);
			\coordinate (F6) at (-1,-0.25);
			
			\node[vertice] at (U) {};
			\node[below] at (U) {$8$};
			\node[vertice] at (A1) {};
			\node[below] at (A1) {$12$};
			\node[vertice] at (B1) {};
			\node[below] at (B1) {$12$};
			\node[vertice] at (A2) {};
			\node[below] at (A2) {$26$};
			\node[vertice] at (B2) {};
			\node[below] at (B2) {$14$};
			\node[vertice] at (A3) {};
			\node[below] at (A3) {$26$};
			\node[vertice] at (B3) {};
			\node[below] at (B3) {$14$};
			\node[vertice] at (C2) {};
			\node[below] at (C2) {$13$};
			\node[vertice] at (C3) {};
			\node[below] at (C3) {$13$};
			
			\draw (A3)--(A1);
			\draw (B1)--(B3);
			\draw (C3)--(A3);
			\draw (A2)--(C2);
			\draw (U)--(A1)--(A2);
			\draw (U)--(B1)--(B2);
			\draw[->] (B2)--(F1);
			\draw[->] (B2)--(F2);
			\draw[->] (B3)--(F3);
			\draw[->] (B3)--(F4);
			\draw[->] (A2)--(F5);
			\draw[->] (A3)--(F6);

			\coordinate (RU) at (4,0);
			\coordinate (RA1) at (5,0);
			\coordinate (RA2) at (6,1);
			\coordinate (RA3) at (6,-1);
			\coordinate (RB1) at (3,0);
			\coordinate (RB2) at (2,1);
			\coordinate (RB3) at (2,-1);
			\coordinate (RC2) at (1,-1);
			\coordinate (RC3) at (7,-1);
			
			\coordinate (RF1) at (1,0.75);
			\coordinate (RF2) at (1,1.25);
			\coordinate (RF3) at (2,-0.25);
			\coordinate (RF4) at (6,-0.25);
			\coordinate (RF5) at (7,1.25);
			\coordinate (RF6) at (7,0.75);
			
			\node[vertice] at (RU) {};
			\node[below] at (RU) {$8$};
			\node[vertice] at (RA1) {};
			\node[below] at (RA1) {$12$};
			\node[vertice] at (RB1) {};
			\node[below] at (RB1) {$12$};
			\node[vertice] at (RA2) {};
			\node[below] at (RA2) {$14$};
			\node[vertice] at (RB2) {};
			\node[below] at (RB2) {$14$};
			\node[vertice] at (RA3) {};
			\node[below] at (RA3) {$26$};
			\node[vertice] at (RB3) {};
			\node[below] at (RB3) {$26$};
			\node[vertice] at (RC2) {};
			\node[below] at (RC2) {$13$};
			\node[vertice] at (RC3) {};
			\node[below] at (RC3) {$13$};
			
			\draw (RA3)--(RA1);
			\draw (RB1)--(RB3);
			\draw (RC3)--(RA3);
			\draw (RB3)--(RC2);
			\draw (RU)--(RA1)--(RA2);
			\draw (RU)--(RB1)--(RB2);
			\draw[->] (RB2)--(RF1);
			\draw[->] (RB2)--(RF2);
			\draw[->] (RB3)--(RF3);
			\draw[->] (RA3)--(RF4);
			\draw[->] (RA2)--(RF5);
			\draw[->] (RA2)--(RF6);

			\coordinate (P1) at (2,-4);
			\coordinate (P2) at (3.5,-4);
			\coordinate (P3) at (5,-4);
			\coordinate (P4) at (6.5,-4);
			\coordinate (P5) at (2,-3);
			\coordinate (P6) at (3.5,-3);
			\coordinate (P7) at (5,-3);
			\coordinate (P8) at (6.5,-3);
			
			\coordinate (Q1) at (-0.5,-4);
			\coordinate (Q2) at (-2,-4);
			\coordinate (Q3) at (-3.5,-4);
			\coordinate (Q4) at (-5,-4);
			\coordinate (Q5) at (-0.5,-3);
			\coordinate (Q6) at (-2,-3);
			\coordinate (Q7) at (-3.5,-3);
			\coordinate (Q8) at (-5,-3);

			\coordinate (Z1) at (-1.25,-3);
			\coordinate (Z2) at (-2.75,-2.75);
			\coordinate (Z3) at (-4.25,-3);
			\node[above] at (Z1) {$c_9$};
			\node[above] at (Z2) {$c_5$};
			\node[above] at (Z3) {$c_1$};
			\coordinate (Z1) at (-1.25,-4);
			\coordinate (Z2) at (-2.75,-4.25);
			\coordinate (Z3) at (-4.25,-4);
			\node[below] at (Z1) {$c_{12}$};
			\node[below] at (Z2) {$c_8$};
			\node[below] at (Z3) {$c_4$};
			\coordinate (Z6) at (-1.75,-3.6);
			\coordinate (Z7) at (-3,-3.7);
			\coordinate (Z8) at (-4.75,-3.6);
			\node[above] at (Z6) {$c_{10}$};
			\node[above] at (Z7) {$c_6$};
			\node[above] at (Z8) {$c_2$};
			\coordinate (Z11) at (-1.85,-3.9);
			\coordinate (Z10) at (-2.5,-3.9);
			\coordinate (Z9) at (-4.85,-3.9);
			\node[above] at (Z11) {$c_{11}$};
			\node[above] at (Z10) {$c_{7}$};
			\node[above] at (Z9) {$c_{3}$};

			\node[vertice] at (P1) {};
			\node[vertice] at (P2) {};
			\node[vertice] at (P3) {};
			\node[vertice] at (P4) {};
			\node[vertice] at (P5) {};
			\node[vertice] at (P6) {};
			\node[vertice] at (P7) {};
			\node[vertice] at (P8) {};
			
			\node[below] at (P1) {$26,6$};
			\node[below] at (P2) {$12,2$};
			\node[below] at (P3) {$12,2$};
			\node[below] at (P4) {$26, 6$};
			
			\node[above] at (P5) {$14,6$};
			\node[above] at (P6) {$12,2$};
			\node[above] at (P7) {$12,2$};
			\node[above] at (P8) {$14,6$};

			\coordinate (P1b) at (2,-3.5);
			\coordinate (P4b) at (6.5,-3.5);
			\coordinate (P5c) at (1,-3.25);
			\coordinate (P5b) at (1,-2.75);
			\coordinate (P8b) at (7.5,-3.25);
			\coordinate (P8c) at (7.5,-2.75);
			\draw[->] (P8)--(P8b) node[right] {$a_{r2}$};
			\draw[->] (P8)--(P8c) node[right] {$a_{r1}$};
			\draw[->] (P4)--(P4b)  node[right] {$a_{r3}$};
			\draw[->] (P1)--(P1b)  node[left] {$a_{l3}$};
			\draw[->] (P5)--(P5b)  node[above] {$a_{l1}$};
			\draw[->] (P5)--(P5c)  node[below] {$a_{l2}$};

			\node[vertice] at (Q1) {};
			\node[vertice] at (Q2) {};
			\node[vertice] at (Q3) {};
			\node[vertice] at (Q4) {};
			\node[vertice] at (Q5) {};
			\node[vertice] at (Q6) {};
			\node[vertice] at (Q7) {};
			\node[vertice] at (Q8) {};
			
			\node[below] at (Q1) {$26,6$};
			\node[below] at (Q2) {$12,2$};
			\node[below] at (Q3) {$12,2$};
			\node[below] at (Q4) {$14,6$};
			
			\node[below right] at (Q5) {$26,6$};
			\node[above] at (Q6) {$12,2$};
			\node[above] at (Q7) {$12,2$};
			\node[above] at (Q8) {$14,6$};
			
			\coordinate (Q1b) at (-0.5,-3.5);
			\coordinate (Q4b) at (-6,-4.25);
			\coordinate (Q4c) at (-6,-3.75);
			\coordinate (Q5b) at (-0.5,-2.5);
			\coordinate (Q8b) at (-6,-3.25);
			\coordinate (Q8c) at (-6,-2.75);
			\draw[->] (Q8)--(Q8b);
			\draw[->] (Q8)--(Q8c);
			\draw[->] (Q4)--(Q4b);
			\draw[->] (Q4)--(Q4c);
			\draw[->] (Q1)--(Q1b);
			\draw[->] (Q5)--(Q5b);
			
			\node[right] at (Q5b) {$a_{r1}$};
			\node[below right] at (Q1b) {$a_{r2}$};
			\node[left] at (Q8b) {$a_{l2}$};
				\node[left] at (Q8c) {$a_{l1}$};
				\node[left] at (Q4c) {$a_{l3}$};
				\node[left] at (Q4b) {$a_{l4}$};

			\draw (P1)--(P2);
			\draw (P2) to[out=30,in=150](P3);
			\draw (P2) to[out=-30,in=-150](P3);
			\draw (P3)--(P4);			
			\draw (P1)--(P6);
			\draw (P3)--(P8);
			\draw (P5)--(P6);
			\draw (P6) to[out=30,in=150](P7);
			\draw (P6) to[out=-30,in=-150](P7);
			\draw (P7)--(P8);
			\draw (P5)--(P2);
			\draw (P7)--(P4);
			
			\draw (Q1)--(Q2);
			\draw (Q2)to[out=150,in=30] (Q3);
			\draw (Q2)to[out=-150,in=-30] (Q3);
			\draw (Q3)--(Q4);
			\draw (Q5)--(Q6);
			\draw (Q6) to[out=150,in=30] (Q7);
			\draw (Q6) to[out=-150,in=-30] (Q7);
			\draw (Q7)--(Q8);
			\draw (Q1)--(Q6);
			\draw (Q3)--(Q8);
			\draw (Q5)--(Q2);
			\draw (Q7)--(Q4);

			\end{tikzpicture}
\caption{}	
		\end{center}
	\label{fig:sss_example}
	\end{figure}

Remark that the Milnor number of the curves defined by $f_{11;00}$ and $f_{10;10}$ is $79$ and their characteristic polynomials are
\[
\Delta=\frac{(t-1)(t^{14}-1)^2(t^{12}-1)^2(t^{26}-1)^2}{(t^{13}-1)^2}, \quad \hbox{ and } \quad \Delta_2=\frac{(t^2-1)^2(t^4-1)}{(t-1)^3}.
\]
Let us compute their associated quadratic forms.  Since the set of multiplicities of the exceptional divisors coincide, we have that  $e=
1092$ for both curves.

For the first case (left in \Cref{fig:sss_example}), we consider the following basis for homology: 
	\begin{align*}
	\sigma_1&= c_1-c_{3}+c_4-c_2,\\ 
	\sigma_2&=c_5-c_6,\\ 
	\sigma_3&=c_6+c_{10}-c_{12}-c_7-c_2+c_1,\\ 
	\sigma_4&=c_7-c_8,\\
	\sigma_5&=c_9-c_{11}+c_{12}-c_{10}\\
	\sigma_6&=-a_{l1}+a_{l2},\\
	\sigma_7&=-a_{l2}+c_2-c_4+a_{l3},\\
	\sigma_8&=-a_{l3}+a_{l4},\\
	\sigma_9&=-a_{l4}+c_4+c_8+c_{12}+a_{r2},\\
	\sigma_{10}&=-a_{r2}-c_{10}+c_9+a_{r1},
	\end{align*}
where all $c_i$ are oriented from left to right, and arrows are oriented with their given orientation in \Cref{fig:sss_example}, and we get that the quadratic form associated to $f_{11;00}$
is given by the matrix 
\[
\left(\begin{array}{rrrrr|rrrrr}
52 & 0 & 26 & 0 & 0 & 0 & -26 & 0 & 13 & 0 \\
0 & 182 & -91 & 0 & 0 & 0 & 0 & 0 & 0 & 0 \\
26 & -91 & 222 & -91 & -14 & 0 & -13 & 0 & -7 & -7 \\
0 & 0 & -91 & 182 & 0 & 0 & 0 & 0 & -91 & 0 \\
0 & 0 & -14 & 0 & 28 & 0 & 0 & 0 & 7 & 14 \\
\hline
0 & 0 & 0 & 0 & 0 & 156 & -78 & 0 & 0 & 0 \\
-26 & 0 & -13 & 0 & 0 & -78 & 182 & -78 & -13 & 0 \\
0 & 0 & 0 & 0 & 0 & 0 & -78 & 156 & -78 & 0 \\
13 & 0 & -7 & -91 & 7 & 0 & -13 & -78 & 231 & -42 \\
0 & 0 & -7 & 0 & 14 & 0 & 0 & 0 & -42 & 98
\end{array}\right)
\]
For the second case (right in \Cref{fig:sss_example}), we fix the following basis for homology:
	\begin{align*}
	\sigma_1&= c_1-c_{3}+c_4-c_2,\\ 
	\sigma_2&=c_5-c_6,\\ 
	\sigma_3&=c_6+c_{10}-c_{12}-c_7-c_2+c_1,\\
	\sigma_4&=c_7-c_8,\\
	\sigma_5&=c_9-c_{11}+c_{12}-c_{10}\\
	\sigma_6&=-a_{l1}+a_{l2},\\
	\sigma_7&=-a_{l2}+c_2-c_4+a_{l3},\\
	\sigma_8&=-a_{l3}+c_4+c_8+c_12+a_{r3},\\
	\sigma_9&=-a_{r3}-c_{10}+c_9+a_{r2},\\
	\sigma_{10}&=-a_{r2}+a_{r1},
	\end{align*}
	and we get that the quadratic form associated to $f_{10;10}$
	is given by the matrix 
\[
\left(
\begin{array}{rrrrr|rrrrr}
40 & 0 & 26 & 0 & 0 & 0 & -20 & 7 & 0 & 0 \\
0 & 182 & -91 & 0 & 0 & 0 & 0 & 0 & 0 & 0 \\
26 & -91 & 222 & -91 & -14 & 0 & -13 & -7 & -7 & 0 \\
0 & 0 & -91 & 182 & 0 & 0 & 0 & -91 & 0 & 0 \\
0 & 0 & -14 & 0 & 40 & 0 & 0 & 7 & 20 & 0 \\
\hline
0 & 0 & 0 & 0 & 0 & 156 & -78 & 0 & 0 & 0 \\
-20 & 0 & -13 & 0 & 0 & -78 & 140 & -49 & 0 & 0 \\
7 & 0 & -7 & -91 & 7 & 0 & -49 & 189 & -42 & 0 \\
0 & 0 & -7 & 0 & 20 & 0 & 0 & -42 & 140 & -78 \\
0 & 0 & 0 & 0 & 0 & 0 & 0 & 0 & -78 & 156
\end{array}
\right)
\]
Notice that these matrices have different  determinants (up to squares) and, hence, the two quadratic  forms $\tilde{Q}$ are not similar (over~$\mathbb{Q}$). The restrictions of $\tilde{Q}$ to $H_1(F)$ are  not similar either because the determinants of the $5 \times 5$ minors corresponding to the first 5 rows and first 5 columns  are different (up to squares) too.

It is worthwhile to mention that Selling reduction of the (positive definite quadratic) Seifert form defined on $H_1(\Sigma, \ZZ) \cap H_1(\sigma, \CC)_1$ was used by Kaenders to 
recover the pairwise intersection multiplicity of the different branches of a plane curve singularity \cite[Theorem 1.4]{Kaen} and thus to distinguish the above examples.   
%
%
\end{example}

\begin{example}\label{ex:DBM} DuBois and Michel showed that the Seifert form is not a complete invariant of the topological type of plane curve singularities, see \cite{DBM2}. Let us consider the 
curves $C_{a,b}$
defined by 
\[
f_{a,b}(x,y)=\left((y^2-x^3)^2-x^{b+6}-4 x^{\frac{b+9}{2}} y\right)
\left((x^2-y^5)^2-y^{a+10}-4xy^{\frac{a+15}{2}}\right),
\]
where $a,b$ are odd integers, i.e. of the form $a=2\alpha+1$ and $b=2\beta+1$
and $b\geq 11$.
The Milnor number of $C_{a,b}$ is 
$75+2(\alpha+\beta)=a+b+73$, and their characteristic polynomials are
\[
\Delta= (t-1)(t^{10}+1)(t^{14}+1)(t^{28+a}+1)(t^{20+b}+1), \quad \hbox{ and } \quad \Delta_2=\frac{t^4-1}{t-1},
\]
see \Cref{fig:dbm}. These data are invariant by the change $(a,b)\mapsto(b-8,a+8)$; the singularities
$C_{a,b}$ and $C_{b-8,a+8}$ are not topologically equivalent if $b\neq a+8$ but their
Seifert forms coincide. 
\begin{figure}[ht]
\centering\begin{tikzpicture}[vertice/.style={draw,circle,fill,minimum size=0.2cm,inner sep=0}]
\foreach \a in {0,...,13}
{
\coordinate (A-\a) at (\a,0);
}
\coordinate (A-14) at (1,-1);
\coordinate (A-15) at (5,-1);
\coordinate (A-16) at (8,-1);
\coordinate (A-17) at (12,-1);
\foreach \a in {1,2,4,5,...,9,11,12,14,15,...,17}
{
\node[vertice] at (A-\a) {};
}
\draw[->] ($.5*(A-2)+.5*(A-3)$)--(A-0);
\draw[dashed] ($.5*(A-2)+.5*(A-3)$)--($.5*(A-3)+.5*(A-4)$) ($.5*(A-9)+.5*(A-10)$)--($.5*(A-10)+.5*(A-11)$);
\draw($.5*(A-3)+.5*(A-4)$)--($.5*(A-9)+.5*(A-10)$) (A-1)--(A-14) (A-5)--(A-15) (A-8) -- (A-16) (A-12) -- (A-17);
\draw[->] ($.5*(A-10)+.5*(A-11)$) -- (A-13);

\foreach \x/\y in {7/8,6/12, 4/30,2/{28+2\alpha}, 14/{29+2\alpha}, 9/22, 11/{20+2\beta}, 17/{21+2\beta}, 15/14, 16/10}
{
\node[below] at (A-\x) {$\y$};
}
\foreach \x/\y in {8/20, 5/28, 1/{58+4\alpha}, 12/{42+4\beta}}
{
\node[above] at (A-\x) {$\y$};
}

\begin{scope}[yshift=-4cm,xshift=1cm]

\foreach \a\b in {0,1,6/4,7/5}
{
\coordinate (A-\a) at (2*\b,0);
}
\foreach \a\b\c in {2/2/1, 3/2/-1, 4/3/1,5/3/-1}
{
\coordinate (A-\a) at (2*\b,\c);
}
\foreach \a in {1,...,6}
{
\node[vertice] at (A-\a) {};
}
\draw[->] (A-1)--(A-0) node[left] {$a_l$} ;
\draw (A-1)-- node[above] {$c_1$} (A-2) (A-1) -- node[below] {$c_2$} (A-3)  (A-4) -- node[above] {$c_7$} (A-6)  (A-5)-- node[below] {$c_8$} (A-6);
\draw[->]  (A-6)--(A-7) node[right] {$a_r$} ;
\draw (A-2) to[out=30, in=150] node[above] {$c_3$} (A-4);
\draw (A-2) to[out=-30, in=-150] node[below] {$c_4$}  (A-4);
\draw (A-3) to[out=30, in=150] node[above] {$c_5$} (A-5);
\draw (A-3) to[out=-30, in=-150] node[below] {$c_6$}  (A-5);

\node[above left] at (A-1) {$58+4\alpha$};
\node[below left] at (A-1) {$14+\alpha$};
\node[above] at (A-2) {$28$};
\node[below] at (A-2) {$3$};
\node[above] at (A-3) {$28$};
\node[below] at (A-3) {$3$};
\node[above] at (A-4) {$20$};
\node[below] at (A-4) {$2$};
\node[above] at (A-5) {$20$};
\node[below] at (A-5) {$2$};
\node[above right] at (A-6) {$42+4\beta$};
\node[below right] at (A-6) {$10+\beta$};
\end{scope}

\end{tikzpicture}

\caption{Dual resolution of $f_{2\alpha+1,2\beta+1}$ and Nielsen-Thurston graph (upper weight is multiplicity and lower weight is genus).}
\label{fig:dbm}
\end{figure}
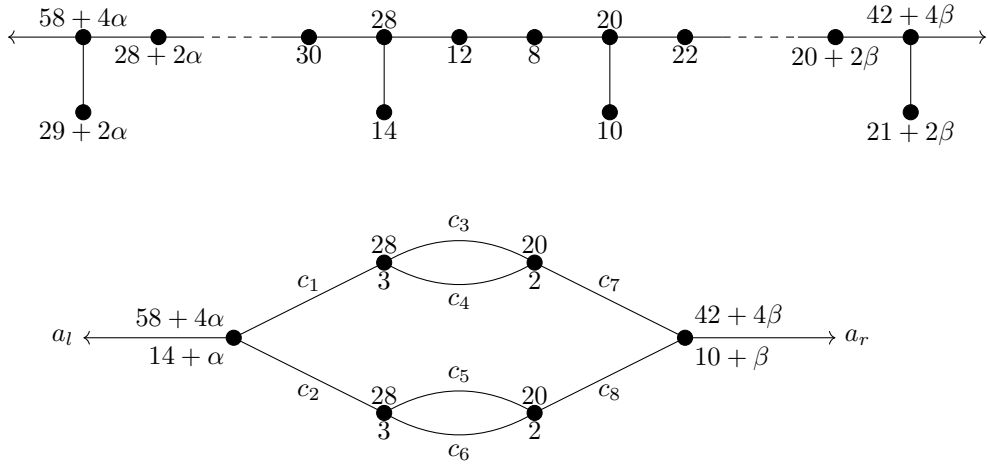

Let us compute the associated quadratic form $\tilde{Q}_{a,b}$. 
Let us fix the following basis for the homology of the Nielsen-Thurston graph
\begin{align*}
\sigma_1&= c_4-c_5,\\ 
\sigma_2&=c_1+c_5+c_8-c_9-c_6-c_2,\\
\sigma_3&=c_6-c_7,\\ 
\sigma_4&=-a_3+c_1+c_4+c_8+a_r, 
\end{align*}
where all $c_i$ are oriented from left to right, and the arrows are given their natural orientation. If we denote $P_2(a,b)={\left(a + 28\right)} {\left(b + 20\right)}$ and
$P_1(a,b)=a+b+48$, then the matrix of the quadratic form
$\tilde{Q}_{a,b}$ in this basis is
\[
\left(\begin{array}{rrrr}
22 \, P_2(a,b) & -11 \, P_2(a,b) & 0 & 11 \,P_2(a,b) \\
-11 \, P_2(a,b) & 46 P_2(a,b) - 280 P_1(a,b) & -11 \, P_2(a,b) & 12 \, P_2(a,b) -  140 P_1(a,b) \\
0 & -11 \, P_2(a,b) & 22 \, P_2(a,b) & 0 \\
11 \, P_2(a,b) & 12 \, P_2(a,b) -  140 P_1(a,b) & 0 & 23 \, P_2(a,b) -  70 P_1(a,b)
\end{array}\right).
\]
Note that the matrices are invariant by the change
$(a,b)\mapsto(b-8,a+8)$.
%
	\end{example}

	\subsection{Monodromy not coming from plane curves}\label{siex}
\mbox{}

Consider the hypersurface singularity defined by $X=\{x y z+x^3-y^3+z^4\}\subset\mathbb{C}^3$. This is an example of a superisolated singularity as were introduced by I. Luengo in \cite{Luen}. In that same paper, Luengo introduced methods to easily compute the self-intersection of the divisors in the resolution of $X$ (see \cite[Lemma 3]{Luen}). Fix any generic linear holomorphic map $ax+by+cz$ and take $f$ to be its restriction to $X$. 
	
	Then the plumbing graph of the link of $X$ is given by the resolution graph of $X$ and the strict transform of $f$ induces a system of multiplicities as in \Cref{fig:graph_non_planecurve}. In this case, since the multiplicity at the only node is $1$, the semistable reduction graph has the same homotopy type. 
	
	\begin{figure}[!ht]
		\includegraphics{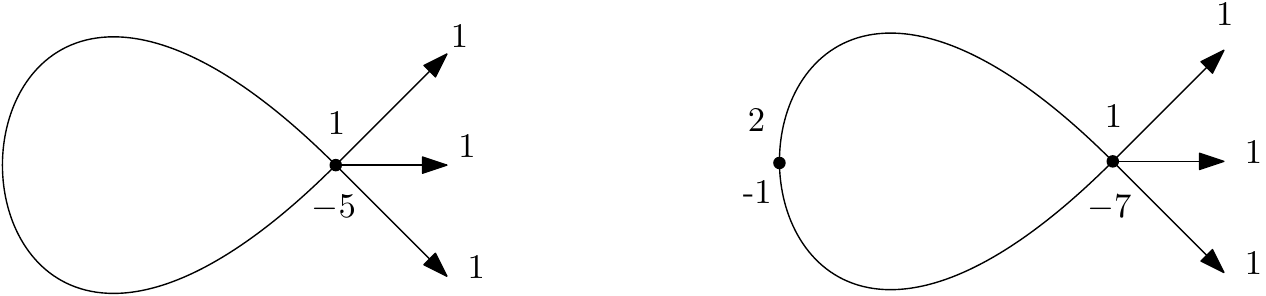}
		\caption{On the left we see the minimal plumbing graph corresponding to the link of $X$, on the right, we see the same graph after blowing up the nodal point of its unique divisor. The negative numbers represent the Euler number of the induced $\mathbb{S}^1$-fiber bundles and the positive numbers represent the multiplicities induced by $f$. Note that in the case of smooth divisors this Euler number is the self-intersection of the corresponding divisors; it is not
			the case for the singular points and this is why we put $-5$ instead of $-3$.}
		\label{fig:graph_non_planecurve}
	\end{figure}
	
	The multiplicity $1$ on the only node indicates that the monodromy on the corresponding piece is the identity. Therefore the monodromy is a multitwist, that is a composition of Dehn twists around disjoint simple closed curves (including  parallel to all boundary components).
	
	Using \Cref{tf} we find that the screw numbers near the boundary components are all $1$ and the screw number corresponding to the loop is also $\frac{1}{2} + \frac{1}{2} =1$.
	
	In the next figure we have drawn a model of the Milnor fiber together with the curves around which the Dehn twists are performed and representatives of a basis for the relative homology.
	\begin{figure}[!ht]
		\includegraphics[scale=0.65]{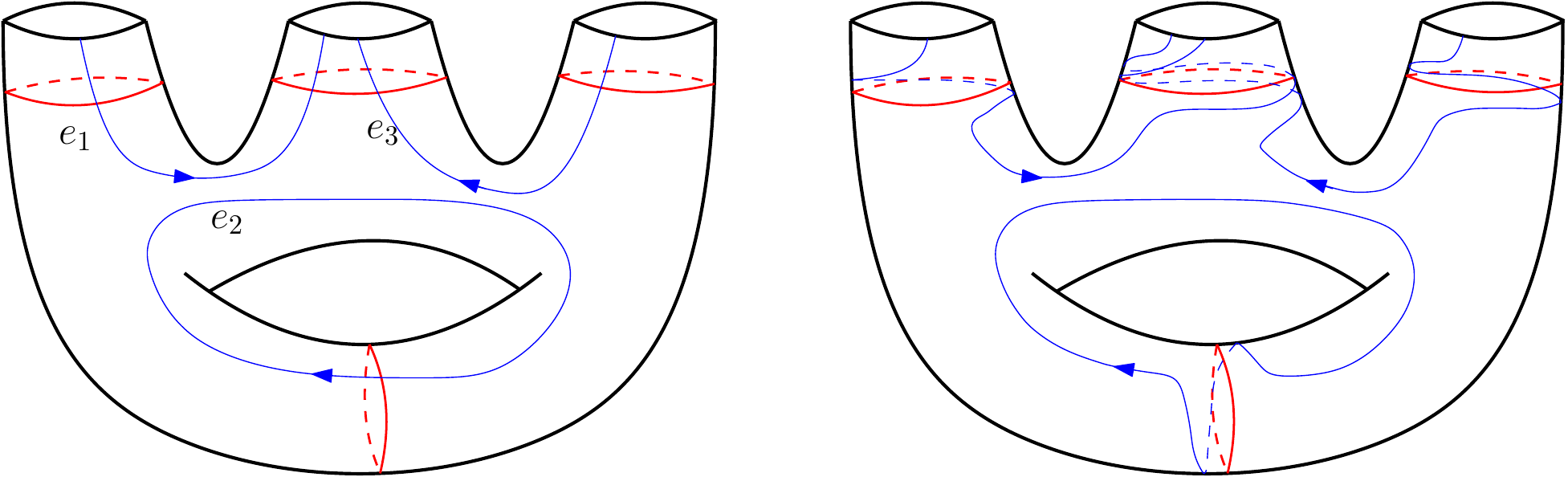}
		\caption{The Milnor fiber of the fibration induced by $f$: a torus minus $3$ disks. The monodromy consists of the composition of one right-handed Dehn twist around each of the red curves. In blue we see generators of the relative homology.}
		\label{fig:milnorfiber_nonplanecurve}
	\end{figure}
	
	With respect to the basis $\{ e_1, e_2, e_3\}$ depicted in \Cref{fig:milnorfiber_nonplanecurve} and the screw numbers, we can compute the associated quadratic form:
\[
\left ( \begin{array}{rrr}
	2&0&1\\
	0&1&0\\
	1&0&2
	\end{array} \right ).
\]
	In particular we observe that it is not an even quadratic form (compare with \Cref{cor:even} and \Cref{cor:even2}).

	\bibliographystyle{amsplain}
	\bibliography{bibliography}
	
\end{document}